%% file: BCH_SVVA_Revision.tex
\documentclass[reqno]{amsart}
\usepackage{a4wide}
\usepackage{amsthm}
\usepackage{amssymb}
\usepackage{algorithm} 
\usepackage{titlesec} 
\usepackage{tikz}
\usepackage{xcolor}
\usepackage{subcaption}
\usepackage{amsmath}
\usepackage{hyperref}
\usepackage{enumitem}

\titleformat{\section}[hang]{ \bf}
  {\thesection}{0.5em}{\large}

  \usepackage[all]{xy}

\titleformat{\subsection}[hang]{\bf}
{\thesubsection}{0.5em}{}

\newtheorem{proposition}{Proposition}[section]
\newtheorem{theorem}[proposition]{Theorem}
\newtheorem{corollary}[proposition]{Corollary}
\newtheorem{lemma}[proposition]{Lemma}
\newtheorem{definition}[proposition]{Definition}
\newtheorem{assumption}[proposition]{Assumption}

\newtheorem{example}[proposition]{Example}
\newtheorem{remark}[proposition]{Remark}

\DeclareMathOperator{\Range}{Im}
\newcommand{\skalp}[1]{\langle #1\rangle}

\def\argmin{ \mathop{{\rm argmin}}}

\newcommand{\epi}{\mathrm{epi}\,}

\newcommand{\p}{\partial}

\newcommand{\xb}{\bar x}
\newcommand{\mv}{\, \vert \,}

\newcommand{\R}{\mathbb{R}}

\newcommand{\rbar}{\overline{\mathbb R}}

\newcommand{\bB}{\mathbb{B}}
\newcommand{\N}{\mathbb{N}}

\newcommand{\ip}[2]{\left\langle #1, #2\right\rangle}

\newcommand{\norm}[1]{\left\Vert #1\right\Vert}

\newcommand{\set}[2]{\left\{#1\,\left\vert\; #2\right.\right\}}
\newcommand{\dist}{{\rm d}}

\newcommand{\bx}{{\bar{x}}}

\def\eps{\epsilon}

\newcommand{\cX}{\mathcal{X}}

\newcommand{\AND}{\ \mbox{ and }\ }
\newcommand{\st}{\ \mbox{s.t.}\ }

\newcommand{\mli}{\delta}

\begin{document}

\title[Sufficient conditions for metric subregularity]{Sufficient conditions for metric subregularity of constraint systems with applications to disjunctive and ortho-disjunctive programs}


\author{Mat\'{u}\v{s} Benko$^{\dag,\ddag}$}
\address{}
\curraddr{}
\email{}
\thanks{$^\dag$Institute of Computational Mathematics, Johannes Kepler University Linz,
A-4040 Linz, Austria, e-mail: benko@numa.uni-linz.ac.at\\
\indent $^\ddag$Faculty of Mathematics, University of Vienna, 1090 Vienna,
Austria}

\author{Michal \v{C}ervinka$^{\flat,\sharp}$}
\address{}
\curraddr{}
\email{}
\thanks{$^\flat$Institute of Economic Studies, Faculty of Social Sciences, Charles University, Opletalova 26, 110 00, Prague 1, Czech Republic, e-mail: michal.cervinka@fsv.cuni.cz\\
\indent $^\sharp$Institute of Information Theory and Automation, Czech Academy of Sciences, Pod Vodarenskou vezi 4, 180 00 Prague 8, Czech Republic, e-mail: cervinka@utia.cas.cz}

\author{Tim Hoheisel$^\diamond$}
\address{}
\curraddr{}
\email{}
\thanks{$^\diamond$Institute of Mathematics and Statistics, McGill University, 805 Sherbrooke St West, Room 1114 
Montr\'eal, Qu\'ebec, Canada H3A 0B9}



\date{\today}

\begin{abstract}
This paper is devoted to the study of the {\em metric subregularity constraint qualification (MSCQ)} for general optimization problems,
with the emphasis on the nonconvex setting.
We elaborate on notions of {\em directional pseudo- and quasi-normality},
recently introduced by Bai et al. (SIAM J. Opt., 2019), which combine the standard approach
via pseudo- and quasi-normality with modern tools of directional variational analysis. 
We focus on applications to {\em disjunctive programs}, where (directional) pseudo-normality is characterized via an  extremal condition.
This, in turn, yields efficient tools to verify pseudo-normality and MSCQ,
which include, but are not limited to,
Robinson's  result on polyhedral multifunctions and  Gfrerer's
{\em second-order sufficient condition for metric subregularity}.
Finally, we refine our study by defining the new class of {\em ortho-disjunctive programs} which comprises prominent optimization problems
such as {\em mathematical programs with complementarity, vanishing or switching constraints.}
\end{abstract}

\maketitle

\input{S1-Intro_v2}
\input{S2-Preliminaries_v2}
\input{S3-CQandMS_v2_1}
\input{S4-DS_v2}
\input{S5-OrthoDisjSets_v2}
\section*{Conclusion}
Building on recently developed directional techniques from variational analysis, 
this paper contains a complex and self-contained study of the metric subregularity constraint qualification (MSCQ)
for broad classes of nonconvex optimization problems including, most importantly, disjunctive programs.
Our findings reveal a common denominator of several prominent sufficient conditions
for MSCQ occurring in the literature. Thus, our study improves understanding of these seemingly independent approaches and provides an additional insight.
Moreover, it offers a wider spectrum of sufficient conditions for MSCQ, including point-based ones,
and consequently also improves existing sufficient conditions.
Furthermore, by introducing the new notion of ortho-disjunctive programs we established an appropriate framework for a unified study
of several nonconvex optimization problems such as mathematical programs with complementarity, vanishing or switching constraints.
These ortho-disjunctive programs hence provide an intriguing area for future research.

%
\section*{Acknowledgments}
\noindent
The authors also thank two anonymous referees for their comments which helped improve the presentation of the material.

\section*{Dedication}
\noindent
The authors would like to dedicate this paper to Helmut Gfrerer in honor of his 60th birthday.

\section*{Funding}
\noindent
The research of the first author was supported by the Austrian Science Fund (FWF) under grant P29190-N32.
The work on the revised version was supported by the FWF grant P32832-N.
The research of the second author was supported 
by the Grant Agency of the Czech Republic (Grant No. 18-04145S).
Part of this work was done while the second author was visiting  McGill University, partially supported by H2020-MSCA-RISE project GEMCLIME-2020 under GA No. 681228.
The research of the third author was supported by an NSERC discovery grant.


%
\input{bibliographyDisjSets}


\end{document}

%% file: S1-Intro_v2.tex
\section{Introduction}\label{sec:Intro}

In this paper we study {\em constraint qualifications (CQs)}
for a general mathematical program (GMP) given by
\begin{equation}\label{eq:GMP}
\min\limits_{x \in \mathbb{R}^n} \, f(x) \quad  \st \quad x \in F^{-1}(\Gamma)=:\cX,
\end{equation}
where $f:\R^n\to \R$ and $F:\R^n\to \R^d$ are continuously differentiable and $\Gamma \subset \R^d$ is closed.
Constraint qualification are regularity conditions on the feasible set of an optimization problem and  play a crucial role  for  stationarity and optimality conditions, sensitivity analysis or exact penalization,  as well as the convergence analysis of numerical algorithms.

At the center of our attention is the  {\em metric subregularity constraint qualification} (MSCQ).
Known also under other monikers such as {\em error bound property} or
{\em calmness constraint qualification}
(in general, calmness is equivalent to metric subregularity of
the inverse mapping),
MSCQ is, to the best of our knowledge,
the weakest known CQ to ensure the full calculus for (limiting) normal cones and subdifferentials, see \cite{HenJouOut02,IoOut08}.
In particular, MSCQ guarantees that local minimizers of \eqref{eq:GMP} are  {\em Mordukhovich (M)-stationary} \cite{HenJouOut02}.
Moreover, MSCQ also yields  {\em exactness} of the penalty function
\begin{equation}\label{eq:PenaltyFunc}
P_\alpha:=f+\alpha \dist_\Gamma\circ F\quad (\alpha>0),
\end{equation}
see, e.g., \cite{Bu91, Bu91.2, Cla 83,KK02a}, which is an important tool for establishing  necessary optimality conditions, as well as for numerical methods \cite{Bu91.2}.

Apart from the area of optimality conditions and exact penalization, MSCQ turns out to be essential also in second-order
variational analysis and closely related areas of stability and sensitivity,
cf. e.g., \cite{BeGfrMor18,BeGfrOut18b,GO3,GfrMo17} and references therein.

The main drawback of MSCQ is the difficulty with efficient verification
of this property.
There exist two main approaches to ensure MSCQ.
The first one makes use of the stronger property of {\em metric regularity},
which is closely related to other concepts such as the {\em Aubin property},
{\em (generalized) Mangasarian-Fromovitz constraint qualification (GMFCQ)},
{\em no nonzero abnormal multiplier constraint qualification (NNAMCQ)}.
Metric regularity can be characterized via co-derivatives \cite[Theorem 9.40]{RoW 98} (known as the {\em Mordukhovich criterion}) or via   graphical derivatives \cite[Theorem 4B.1]{DoRo14}.
For more information as to metric regularity we refer to
the monographs \cite{DoRo14,Io17,KK02,Mo18,RoW 98}.

The second approach corresponds to Robinson's famous result on polyhedral multifunctions \cite[{Proposition 1}]{Rob81}
and is, in turn, restricted to this special case.
There are many situations, however,
in which metric subregularity is provably satisfied,
yet can not be detected by either of these approaches.
%
%
%
%

%
%

Therefore, a lot of attention has been given to conditions that lie between metric regularity and metric subregularity.
A very common approach is to provide sufficient conditions for subregularity/calmness in terms of various derivative-like objects
\cite{FabHenKruOut10,HenJouOut02,HeO 05,IoOut08,Kru15a,Penot10,WuY 03}.
An exception is the early and very interesting attempt by Klatte and Kummer \cite[Theorem 3.6]{KK02a}, where calmness of an intersection mapping is studied.
For further details on metric subregularity and related notions as well as more bibliographical pointers on the topic, we refer to the paper and the textbook by Dontchev and Rockafellar \cite{DoRo04,DoRo14}
and to the textbook by Ioffe \cite{Io17}.


We will further focus on the following two strategies:
the first one is obtained by the {\em pseudo-} and {\em quasi-normality},
first introduced for nonlinear programming in \cite{BeO 02},
and later extended to MPCCs in \cite{KaS 10, YeZ 14} as well as to general programs of the form \eqref{eq:GMP}, see \cite{GuoYeZhang13}.
The second one, based on the directional approach recently developed
by Gfrerer and co-authors \cite{BeGfrOut18,Gfr13a,Gfr13b,Gfr14a,GfrKla16,GO3},
was established and utilized in \cite{Gfr13a,Gfr13b,GfrKla16} under the name
{\em first/second-order sufficient condition for metric subregularity} (FOSCMS/SOSCMS).
FOSCMS can be viewed as a directional, less restrictive counterpart of the Mordukhovich criterion.
The main advantage of these conditions is their {\em point-based} nature,
which makes it possible to verify them efficiently.
The point-based character of these conditions can be justified by the existence of suitable
calculus rules for the (directional) limiting normal cones, despite the fact that these objects are defined with information taken from
the neighborhood by using a limiting process.

In this paper, we synthesize
the concepts of pseudo- and quasi-normality with the above mentioned  directional approach,
which also  serves as our main workhorse throughout the paper.
Hence, we study constraint qualifications called
{\em directional pseudo-/quasi-normality},
which are milder than both pseudo-/quasi-normality and FOSCMS, and imply MSCQ (cf. Theorems \ref{th:Nonexist} and \ref{The : MainDQNtoMS}).

We would like to point out, that despite working on this combined approach independently of
Bai, Ye, and Zhang \cite{BYZ19}, the exact same definitions of directional pseudo- and quasi-normality
were first published in said paper.
Here, we present alternative or simpler proofs of certain common results using
different techniques, which can further illuminate these novel tools for the reader.
Moreover, we present a thorough investigation of applicability of these new CQs,
which goes beyond the material in \cite{BYZ19}.

Although the core material of our study is valid for general programs \eqref{eq:GMP} with an arbitrary closed set $\Gamma$,
we are particularly interested in situations when $\Gamma$ is not convex.
Optimization problems with inherently nonconvex structures induced by imposing logical or combinatorial conditions on otherwise smooth or convex data \cite{Sch 04}
has been of  increasing interest in recent years.
Among the prominent examples are  {\em mathematical programs with complementarity constraints} (MPCCs) \cite{LuPaRa96,OutKoZo98},
or  {\em mathematical programs with vanishing constraints} (MPVCs)  \cite{Ho09}, etc.
For these optimization problems there are several applications in the natural and social sciences,
economics and engineering. Moreover, they are very challenging from both a theoretical and numerical perspective.
More examples of such programs are discussed in Section 4, where we apply our results to {\em disjunctive programs}
in which $\Gamma$ is a finite union of polyhedra.
In Section 5 we introduce the new notion of {\em ortho-disjunctive programs}.
Ortho-disjunctive programs are disjunctive programs with an additional product structure of $\Gamma$
which allows us to address some issues that cannot be resolved in the general disjunctive setting.
Both disjunctive and ortho-disjunctive programs provide a unified framework for the above mentioned particular problem classes.

The main contributions of the paper are as follows:

\begin{itemize}[leftmargin=*]

\item \textit{Pseudo-normality for disjunctive programs:} For disjunctive programs,
we observe that pseudo-normality can be cast in a simpler way which is, in fact,
 a proper extension of the definition of pseudo-normality that has already been used for  NLPs and  MPCCs in the literature.
 This new definition, however, reveals an interpretation of pseudo-normality via  an extremal condition, see \eqref{eq : PNdef},
which is neither visible from the general definition for \eqref{eq:GMP} nor from the specially tailored ones for NLPs and MPCCs, respectively.
This extremal condition then yields efficient tools to verify pseudo-normality.
Indeed, apart from recovering the Robinson's result and the Gfrerer's SOSCMS,
employing higher-order analysis yields a variety of new milder point-based sufficient conditions for pseudo-normality and MSCQ, see Section \ref{sec:Higher}.

\item \textit{Quasi-normality for ortho-disjunctive programs:}
A similar approach as the one to pseudo-normality can be made for (directional) quasi-normality
if one moves from the disjunctive to even more specialized ortho-disjunctive setting,
designed to utilize an underlying product structure exhibited by the standard examples of disjunctive
programs (MPCCs, MPVCs).
The corresponding extremal condition characterizing quasi-normality leads to a
surprising connection between quasi-normality and multi-objective optimization.
Again, sufficient conditions of second- or higher-order are readily available.

\item \textit{PQ-normality:}
In Section 3 we  established  the new notion of (directional) PQ-normality,
which includes both pseudo- and quasi-normality as extreme cases.
This unified notion puts us in a position to better understand  and  to exploit   certain product structures for which neither
quasi- nor  pseudo-normality  is suitable.

\end{itemize}

The rest of the paper is organized as follows. In Section 2 we present some preliminary results
and notions from variational analysis as well as key results regarding constraint qualifications.
Section 3 contains fundamental results of our study dealing with
CQs for the general program \eqref{eq:GMP}.
In Section 4, we study disjunctive programs
and obtain full results on pseudo-normality.
Section 5 deals with disjunctive programs with additional product structures
often present in the problems of interest (MPCCs, MPVCs, etc.).
In particular, the notion of ortho-disjunctive programs is introduced and
complete results on quasi-normality are obtained.
 \\
 \\
{\em Notation:} Most of the notation used is standard:
 The closed ball in $\R^n$ with center at $x$ and radius $r$ is denoted by $\bB_r(x)$
 and we use $\bB:=\bB_1(0)$ for the closed unit ball.
 The extended real line is given by $\rbar:=\R\cup\{\pm \infty\}$. For $f:\R^n\to \rbar$ its {\em epigraph} is given by
 $\epi f:=\set{(x,\alpha)\in \R^n\times \R}{f(x)\leq \alpha}$.
For a nonempty set $S\subset \R^n$ we define the (Euclidean) {\em distance function} $\dist_S:\R^n\to \R$ through
$\dist_S(x):= \inf_{y\in S}\|x-y\|$.
The {\em projection mapping} $P_S:\R^n\rightrightarrows S$ associated with $S$ is defined by
$P_S(x):=\argmin_{y\in S}\|x-y\|$.
We write $\{x_k\}$ for a sequence of scalars and $\{x^k\}$ for a sequence of vectors.
For a mapping $F:\R^n\to \R^m$ its  {\em Jacobian}  at $\bar x$  is denoted by $\nabla F(\bar x)$.  In particular, for $f$   $f:\R^n\to \R$, the Jacobian $\nabla f(\bar x)$  at $\bar x$  is a row vector, and we denote its  {\em Hessian} at $\xb$ by $\nabla^2 f(\xb)$. Moreover, for $\lambda \in \R^m$ the scalarized function $\ip{\lambda}{F}:\R^n\to \R$ is given by $\ip{\lambda}{F}(x)=\lambda^TF(x)$.
Note that for $u \in \R^n$ we have
$\nabla \ip{\lambda}{F}(\xb)^T u = \ip{\lambda}{\nabla F(\xb)u}$
and we often use the latter notation.
For a matrix $A\in \R^{m\times n}$, its  {\em range} or {\em image} is $\Range A:=\set{Ax}{x\in \R^n}$.  For some vector $v\in \R^n$ we set
$\R_+v:=\set{t v}{t\geq 0}$ and $\R_-v:=\set{t v}{t \leq  0}$.

%% file: S2-Preliminaries_v2.tex
\section{Preliminaries}

This section is divided into two parts.
First, we introduce some basic notions from variational analysis.
The second part is devoted to constraint qualifications for the general mathematical program \eqref{eq:GMP}.

\subsection{Variational analysis}

Given a closed set $C \subset \R^n$ and  $z \in C$,   the  {\em  tangent cone} to $C$ at $z$ is defined by
  \[T_{C}(z) := \set{ d \in \R^n}{\exists \{d^k\} \to d, \{t_k\} \downarrow 0:\;
  z + t_k d^k \in C \;(k\in \N) }.
   \]
  The {\em regular normal cone} to $C$ at $z$  is given  as the polar cone of the tangent cone, i.e.
  \[
    \widehat N_{C}(z) :=  \{ z^* \in \R^n \mv \skalp{z^*,d} \leq 0 \;(d \in T_{C}(z))\}.
  \]
  The {\em limiting normal cone} to $C$ at $z$ is given by
  \[
    N_{C}(z) := \set{ z^* \in \R^n}{ \exists \{\tilde z^k\} \to z^*, \{z^k\} \to z :\; z^k \in C, \tilde z^k \in \widehat N_{C}(z^k)
    \;(k\in \N)}.
  \]
If $z \notin C$ we set $T_{C}(z) := \widehat N_{C}(z) := N_{C}(z) := \emptyset$. Observe that  $\widehat N_{C}(z) \subset N_{C}(z)$ holds.    In case $C$ is a convex set, regular and limiting normal cone coincide with the classical normal cone of convex analysis, i.e.,
  \begin{equation} \label{eq : convexNC}
    \widehat N_{C}(z) = N_{C}(z) = \set{z^* \in \R^n}{ \skalp{z^*,v - z} \leq 0 \ (v \in C)},
  \end{equation}
  and we will use the notation $N_{C}(z)$ in this case.
  Finally, given a direction $d \in \R^n$, the {\em limiting normal cone} to $C$ at $z$ {\em in direction} $d$ is defined by
\begin{equation*}\label{eq : NCdirdef}
  N_{C}(z;d):=\set{z^* \in\R^n}{\exists \{t_k\} \downarrow 0,\ \{d^k\}\to d,\ \{\tilde z^k\} \to z^*:\
  \tilde z^k \in \widehat N_{C}(z+t_k d^k) \ (k\in \N)}.
\end{equation*}
Note that, by definition, we have $N_{C}(z;0)=N_{C}(z)$.
Furthermore, observe that $N_{C}(z;d) \subset N_{C}(z)$ for all $d\in \R^n$ and $N_{C}(z;d)=\emptyset$
if $d \notin T_{C}(z)$.

For $f:\R^n\to \rbar$ and $\bar x$ such that $f(\bar x)$ is finite (hence $(\bar x,f(\bar x))\in \epi f)$) the  sets
\[
\hat \partial f(\xb):=\big\{\xi\in\R^n\;\big|\;(\xi^*,-1) \in \widehat N_{{\rm\small epi}\,f}\big(\xb,f(\xb)\big)\big\}, \quad
\partial f(\xb):=\big\{\xi\in\R^n\;\big|\;(\xi^*,-1) \in N_{{\rm\small epi}\,f}\big(\xb,f(\xb)\big)\big\}
\]
denote the {\em regular} and {\em limiting subdifferential} of $f$ at $\xb$, respectively.
 Observe that, in particular, for the indicator function of a set $C\in \R^{n}$, given by
 \[
 \delta_C:x\mapsto \left\{\begin{array}{ll} 0 & \text{if} \; x \in C,\\
 +\infty& \text{else},
 \end{array}\right.
 \]
 we have
$
\hat\p\delta_C=N_{C}\AND \p\delta_C=N_{C}.
$
The distance function enjoys a rich subdifferential calculus briefly summarized in the next result.

\begin{proposition} [Subdifferentiation of distance function]\label{prop:Dist} Let $S\subset \R^d$ be closed and $F:\R^n\to \R^d$ continuously differentiable. Then the following hold:
\begin{itemize}
\item[(i)] (\cite[Example 8.53]{RoW 98}) $\partial \dist_S(y)=  \left\{\begin{array}{ll} N_S(y)\cap \bB & \text{if}\; y\in S,\\
\frac{ y-P_S(y)}{\dist_S(y)}& \text{if}\; y\notin S;
\end{array}\right.$
\item[(ii)] (\cite[Theorem 10.6]{RoW 98}) $\partial (\dist_S\circ F)(x)\subset \nabla F(x)^T \partial \dist_S(F(x)).$
%
\end{itemize}
\end{proposition}
%
\if{
\noindent
The following elementary observation, which  is readily derived  from first-order optimality conditions, is stated,  e.g.,  in \cite[Example 6.16]{RoW 98}.

\begin{lemma}\label{lem:ProjNorm} Let $C\subset \R^n$ be closed and $\bar x \in P_C(x)$. Then
\[
t(x-\bar x)\in \widehat N_C(\bar x)\quad (t\geq 0).
\]
\end{lemma}

\noindent
We will make use of {\em Ekeland's variational principle} \cite{Eke74}, which we provide for the reader's convenience in the form
given in \cite[Proposition 1.43]{RoW 98}.

\begin{proposition}[Ekeland's variational principle]\label{prop:Ekeland}
Let $f: \R^n \to \bar \R$ has closed epigraph $\epi f$ with $\inf_{x \in \R^n} f(x)$ finite, and let $\xb$
be an $\varepsilon$-minimizer of $f$ for some $\varepsilon>0$, i.e., $f(\xb) \leq \inf_{x} f(x)+\varepsilon$. Then for any $\delta>0$ there exists a point $\tilde x$ such that
\[
\norm{\tilde x - \xb} \leq \varepsilon/\delta,\quad f(\tilde x) \leq f(\xb)\AND
\argmin\left\{f + \delta \norm{(\cdot)- \tilde x}\right\} = \{\tilde x\}.
\]
\end{proposition}
}\fi

\subsection{Constraint qualifications}

The purpose of this section is to recall several well-established CQs for the general program \eqref{eq:GMP}
and to highlight some basic relations between them. We commence with the CQ that is most important to our study.

\begin{definition}[MSCQ]\label{def:MSCQ}
Let $\bar x$ be feasible for \eqref{eq:GMP}. We say that the {\em metric subregularity constraint qualification (MSCQ)}
holds at $\bar x$ if there exists a neighborhood $U$ of $\bar x$ and $\kappa>0$ such that
\[
\dist_{\cX}(x)\leq \kappa \dist_{\Gamma}(F(x))\quad (x\in U).
\]
\end{definition}
Note that MSCQ is exactly metric subregularity in the set-valued sense
of the {\em feasibility mapping} for the  constraint system $\cX=F^{-1}(\Gamma)$
 which is given by  $M(x) := F(x) - \Gamma$, see e.g. \cite{GfrKla16}.

The stronger property of metric regularity holds for $M$ around $(\bar x,0)$ if and only if
 there are neighborhoods $U$ of $\bar x$ and $V$ of $0$ and $\kappa>0$ such that
\[
\dist_{M^{-1}(y)}(x) \leq \kappa \dist_{M(x)}(y) = \kappa \dist_{\Gamma}(F(x)-y) \quad ((x,y) \in U \times V).
\]
It is well-known that metric regularity of a multifunction is equivalent to the {\em Aubin property}
of the inverse multifunction \cite[Theorem 9.43]{RoW 98}.
Applying the Mordukhovich criterion to the feasibility mapping $M$ yields a condition that
there is no nonzero multiplier $\bar \lambda\in N_{\Gamma}(F(\bar x))$ such that
\begin{equation}\label{eq:MStat}
\nabla F(\bar x)^T\bar \lambda=0,
\end{equation}
which is often known as {\em generalized Mangasarian-Fromovitz constraint qualification (GMFCQ)} at $\xb$.
In the rest of the paper, we mainly stick to the GMFCQ terminology, but sometimes
refer to GMFCQ also as the Mordukhovich criterion.
Thanks to the calculus rules for limiting normal cones and subdifferentials,
the Mordukhovich criterion often provides an efficient tool for verifying  metric regularity.
There are still plenty of situations, however, where GMFCQ is not fulfilled but MSCQ is.
It is therefore an important and worthwhile endeavor to fill the gap between GMFCQ and MSCQ,
ideally with verifiable conditions at that.
Let us proceed with the next list of constraint qualifications for \eqref{eq:GMP}
relevant for our study, see, e.g., \cite{GuoYeZhang13,GfrKla16}.

\begin{definition}[Constraint qualifications]\label{def:CQ} Let $\bar x\in\cX$ be feasible for \eqref{eq:GMP}.We say that
\begin{itemize}
\item[(i)] \underline{\em pseudo-normality} holds at $\bar x$ if there is no nonzero $\bar \lambda\in N_{\Gamma}(F(\bar x))$ such that \eqref{eq:MStat} holds and that satisfies the following condition:
There exists a sequence $\{(x^k,y^k,\lambda^k)\in \R^n\times \Gamma\times \R^d\}\to (\bar x,F(\bar x), \bar \lambda)$ with
\[
\lambda^k\in \widehat N_{\Gamma}(y^k)\; \AND \ip{\bar \lambda}{F(x^k)-y^k}>0\quad (k\in \N);
\]
\item[(ii)]\underline{\em quasi-normality} holds at $\bar x$ if there is no nonzero $\bar \lambda \in N_{\Gamma}(F(\bar x))$ such that \eqref{eq:MStat} holds and that satisfies the following condition:
There exists a sequence $\{(x^k,y^k,\lambda^k)\in \R^n\times \Gamma\times \R^d\}\to (\bar x,F(\bar x),\bar \lambda )$ with
\[
\lambda^k\in \widehat N_{\Gamma}(y^k)\; \AND \bar \lambda_i (F_i(x^k)-y_i^k)>0 \quad\text{if}\quad \bar \lambda_{i}\neq 0\quad (k\in\N);
\]
\item[(iii)] \underline{\em first-order sufficient condition for metric subregularity (FOSCMS)} holds at $\xb$
if for every $0 \neq u \in \R^n$ with $\nabla F(\xb) u \in T_{\Gamma}(F(\xb))$ one has
\begin{equation*}
\nabla F(\xb)^T \lambda = 0, \ \lambda \in N_{\Gamma}(F(\xb);\nabla F(\xb) u) \ \Longrightarrow \
\lambda = 0;
\end{equation*}
\item[(iv)] \underline{\em second-order sufficient condition for metric subregularity (SOSCMS)} holds at $\xb$
if $F$ is twice differentiable at $\xb$, $\Gamma$ is the union of finitely many convex polyhedra, and for every $0 \neq u \in \R^n$ with $\nabla F(\xb) u \in T_{\Gamma}(F(\xb))$ one has
\begin{equation*}
\nabla F(\xb)^T \lambda = 0, \ \lambda \in N_{\Gamma}(F(\xb);\nabla F(\xb) u), \
u^T \nabla^2 \ip{\lambda}{F}(\xb) u \geq 0 \ \Longrightarrow \
\lambda = 0.
\end{equation*}
\end{itemize}
\end{definition}

We point out that imposing that the (nonexisting) multiplier $\bar \lambda$  is in $N_\Gamma (F(\bar x))$ in the definition of pseudo-/quasi-normality is, clearly, redundant, since
it follows from $\lambda^k\in \widehat N_{\Gamma}(y^k)$.
Nevertheless, in order to be consistent with the literature
and to emphasize the connection to GMFCQ and other CQs, we stick to the original definition.
In particular, it is obvious from the definition that GMFCQ
implies both pseudo- and hence quasi-normality.
The concepts of pseudo- and quasi-normality are well-established in the literature.
Note that in \cite{GuoYeZhang13}, the condition $\lambda^k\in \widehat N_{\Gamma}(y^k)$ in (i) and (ii)
is replaced by $\lambda^k\in N_{\Gamma}(y^k)$.
In order to see that no difference arises, consider the following elementary lemma
which follows readily from the definitions of continuity and of the limiting normal cone, respectively.

\begin{lemma}\label{lem:Aux}
Let $\Gamma \subset \R^d$ be closed,   $y \in \Gamma, \lambda \in N_{\Gamma}(y)$ and let $a: \R^d \times \R^d \to \R^q$ be continuous. Then for every  $\varepsilon > 0$ there exist
$\tilde y \in \Gamma$ and $\tilde \lambda \in \widehat N_{\Gamma}(\tilde y)$ such that $\norm{a(\tilde y, \tilde \lambda) - a(y, \lambda)} < \epsilon$.
\end{lemma}

\begin{corollary}\label{cor:Aux} Under the assumptions of Definition \ref{def:CQ} let   $\{(x^k,y^k,\lambda^k)\in \R^n\times \Gamma\times \R^d\}\to (\bar x,F(\bar x), \bar \lambda)$. Then the following hold:
\begin{itemize}
\item[(i)] If $\lambda^k\in N_{\Gamma}(y^k)$ and $ \ip{\bar \lambda}{F(x^k)-y^k}>0$ for all $k\in \N$ then there exists  $\{(\tilde y^k,\tilde \lambda^k)\}\to (F(\bar x),\bar \lambda) $ such that  $\tilde \lambda^k\in \widehat N_\Gamma(\tilde y^k)$ and $\ip{\bar \lambda}{F(x^k)-\tilde y^k}>0$ for all $k\in \N$.

\item[(ii)]  If $\lambda^k\in  N_{\Gamma}(y^k)$ and $\bar \lambda_i (F_i(x^k)-y_i^k)>0  \;(i:\bar \lambda_{i}\neq 0)$ for all $k\in \N$ then there exists $\{(\tilde y^k,\tilde \lambda^k)\}\to (F(\bar x),\bar \lambda) $ such that
$\tilde \lambda^k\in \widehat N_\Gamma(\tilde y^k)$ and  $\bar \lambda_i (F_i(x^k)-\tilde y_i^k)>0  \;(i:\bar \lambda_{i}\neq 0)$ for all $k\in \N$.

\end{itemize}

\end{corollary}
\begin{proof} We only prove part (i); part (ii) can be shown analogously. To this end, define the continuous maps
\[
a_{k}:(y,\lambda)\mapsto (y,\lambda,  \ip{\bar \lambda}{F(x^k)-y})\quad (k\in \N),
\]
and set $\eps_k:=\min \left \{ \frac{1}{k}, \frac{1}{2}\ip{\bar \lambda}{F(x^k)-y^k} \right \}$. Applying Lemma \ref{lem:Aux} then generates the desired sequences.
\end{proof}

\noindent
Corollary \ref{cor:Aux} guarantees that using $\lambda^k\in \widehat N_{\Gamma}(y^k)$ instead of
$\lambda^k\in N_{\Gamma}(y^k)$ in the  definition of pseudo- and quasi-normality does not matter.
We note that this is also true for the directional versions of these CQs to be established in Definition \ref{def:DirCQ}.

The obvious drawback of pseudo- and quasi-normality is that they
are expressed via sequences, which makes it is quite difficult to check their validity and apply them.
Another way of relaxing GMFCQ is provided by FOSCMS and SOSCMS, which can be easier to verify
due to the calculus for directional limiting objects  \cite{BeGfrOut18}.

In order to simplify the notation, given $\bar x$ feasible for \eqref{eq:GMP}, we define
\begin{equation}\label{eq : LamdbaDef}
 \Lambda^0(\xb;u) := \ker \nabla F(\xb)^T \cap N_{\Gamma}(F(\xb);\nabla F(\xb)u) \quad (u\in \R^n)
 \end{equation}
and set
\[
\Lambda^0(\xb) := \Lambda^0(\xb;0)=\ker \nabla F(\xb)^T \cap N_{\Gamma}(F(\xb)),
\]
i.e., the directional normal cone is replaced by the standard one.
With these conventions, GMFCQ at $\bar x$ reads
\[
\Lambda^0(\xb) =\{0\},
\]
while FOSCMS now reads
\[
\Lambda^0(\xb;u)=\{0\}\quad (u:\nabla F(\xb)u\in T_{\Gamma}(F(\bar x))).
\]
The fact that GMFCQ implies FOSCMS is clear from the inclusion
\[
N_{\Gamma}(F(\xb);\nabla F(\xb)u)\subset N_{\Gamma}(F(\xb))\quad (u\in \R^n).
\]
The following example shows that this implication can be strict. In addition, it also illustrates that MSCQ is strictly weaker than quasi-normality,
cf. Proposition \ref{The : CQsToMS}(i).

\begin{example}\label{ex: DirCQ}
Let  $\Gamma :=\{y\in \R^2 \mv y_2 \geq \vert y_1 \vert\}\subset \R^2, \;F: \R \to \R^2,\;F(x) := (x,-x^2)^T$ and set $\xb := 0$.
Clearly $\nabla F(\xb)=(1,0)^T$ and $N_{\Gamma}(F(\xb))=\{y\in \R^2 \mv y_2 \leq -\vert y_1 \vert\}$,
hence $0 \neq \lambda:=(0,-1)^T \in \Lambda^0(\xb)$
and the Mordukhovich criterion (GMFCQ) is violated at $\bar x$.

Moreover, setting $x_k := 1/k$, $y^k:=F(\xb)=(0,0)^T$ and $\lambda^k := \lambda=(0,-1)^T$
we obtain $\lambda_{2}(F_2(x_k) - y^k_2)=-1(-1/(k^2)) > 0$,
showing that also quasi-normality is violated at $\bar x$.

On the other hand, since $N_{\Gamma}(F(\xb);\nabla F(\xb) u) = \emptyset$ for all $u \neq 0$,
FOSCMS and hence MSCQ are satisfied at $\bar x$.
\end{example}

\noindent
We point out that the set $\Gamma$ in Example \ref{ex: DirCQ} is convex, thus illustrating that even in the convex case one may not be able to
verify MSCQ using the non-directional conditions (GMFCQ, pseudo- and quasi-normality), but one may invoke a directional one (here FOSCMS).

Although the directional conditions FOSCMS and SOSCMS are similar in flavor, we point out that SOSCMS is only applicable in the case where $\Gamma$
has {\em disjunctive} structure. In this setting, there is yet another condition
due to Robinson \cite{Rob81} that ensures MSCQ.

The following proposition summarizes several important sufficient conditions for MSCQ,
other than GMFCQ, which have already been established in the literature and that are important to our study.
We point out, however, that the validity of these results will be a simple corollary of our refined analysis in Section \ref{sec:NewCQ}.

\begin{proposition}[Sufficient conditions for MSCQ]\label{The : CQsToMS}
 Let $\bar x$ be feasible for \eqref{eq:GMP}. Then under either of the following conditions MSCQ
holds at $\bar x$.
\begin{itemize}
\item[(i)] ({\cite[Theorem 5.2]{GuoYeZhang13}})quasi-normality (or pseudo-normality) holds at $\bar x$;
\item[(ii)] (\cite[Corollary 1]{GfrKla16}) FOSCMS holds at $\xb$;
\item[(iii)] (\cite[Corollary 1]{GfrKla16}) SOSCMS holds at $\xb$;
\item[(iv)] ({\cite[{Proposition 1}]{Rob81}}) $F$ is affine and $\Gamma$ is the union of finitely many convex polyhedra.
\end{itemize}
\end{proposition}

As we can see, two of these conditions are applicable for the general program \eqref{eq:GMP} and are strictly milder than GMFCQ.
The other two are restricted to the special structure of disjunctive constraints and hence are in general not
comparable with GMFCQ.
Interestingly, all four conditions are mutually incomparable and were obtained by different approaches.
The only available comparison is for the disjunctive constraints, where FOSCMS clearly implies SOSCMS.

We will refer to (iv) in Proposition \ref{The : CQsToMS} as {\em Robinson's result}.
We point out, however, that \cite[{Proposition 1}]{Rob81} in fact contains a stronger statement.

\if{
\subsection{Exact penalization under MSCQ}
\noindent
Let us briefly discuss the role of  MSCQ in the context of exact penalization.
Note that the general mathematical program \eqref{eq:GMP}
is equivalent to the unconstrained (but extended real-valued) problem
\begin{equation}\label{eq:GMPUncon}
\min\; f(x)+\delta_\Gamma (F(x)).
\end{equation}

\noindent
A natural approximation for \eqref{eq:GMPUncon} (and hence \eqref{eq:GMP}) is given by minimization of the following penalty function
\begin{equation}\label{eq:PenaltyFunc}
P_\alpha:=f+\alpha \dist_\Gamma\circ F\quad (\alpha>0),
\end{equation}
 which is a classical technique employed to tackle program \eqref{eq:GMP}, see, e.g.,   \cite{Bu91, Bu91.2, GuoYeZhang13, HoKaOut10, HuRa04, KaS 10, LuPaRaWu96, QHYM18}.

The crucial issue is the exactness of the penalty function,
which holds true under MSCQ as is stated in the following theorem.
The proof essentially coincides with the proof of Theorem \cite[Theorem 4.5]{KaS 10}, but we provide it for the sake of completeness and also to realize that the special structure underlying in \cite{KaS 10} is not needed at all.

Note also that for general programs the following result was first
established in \cite[Proposition 3.5]{HoKaOut10} and it was based
on results by Burke \cite[Theorem 1.1]{Bu91} and Clarke \cite[Proposition 6.4.3]{Cla 83}.

\begin{theorem}\label{cor:ExPenMSCQ}
Let $\bar x$ be a local minimizer of \eqref{eq:GMP} such that
MSCQ holds at $\bar x$. Then the penalty function $P_\alpha$ from \eqref{eq:PenaltyFunc} is exact at $\bar x$, i.e., $\bar x$ is a local minimizer of $P_\alpha$ for all $\alpha >0$ sufficiently large. In particular, $\bar x$ is an M-stationary point of \eqref{eq:GMP}.
\end{theorem}
\begin{proof} By MSCQ at $\bar x$ there exist $\delta,\kappa>0$ such that
\[
\dist_{\cX}(x)\leq \kappa \dist_{\Gamma}(F(x))\quad (x\in \bB_{\delta}(\bar x)).
\]
As $\bar x$ is a local minimizer of $f$ over $\cX$, we can choose $0<\varepsilon<\delta/2$ such that $\bar x\in \argmin_{x\in \bB_{2\varepsilon}(\bar x)\cap\cX} f(x)$. Since $f$ is locally Lipschitz, by compactness, $f$ is $L$-Lipschitz on $\bB_{2\varepsilon}(\bar x)$ for some $L>0$. Now let  $x\in \bB_{\varepsilon}(\bar x)$. In particular, we find $y\in P_{\cX}(x) \subset \cX\cap \bB_{2\varepsilon}(\bar x)$, hence  it follows that
\[
f(\bar x)\leq f(y)\leq f(x)+L\|y-x\|=f(x)+L\dist_{\cX}(x)\leq f(x)+\kappa L \dist_{\Gamma}(F(x)).
\]
This shows that $\bar x$ is a local minimizer of $f+\kappa L\dist_{\Gamma}\circ F$
and the exactness of $P_\alpha$ follows.
Moreover, applying a nonsmooth Fermat's rule (cf. \cite[Theorem 10.1]{RoW 98}), invoking \cite[Exercise 10.10]{RoW 98} and Proposition \ref{prop:Dist} yields
\[
0\in \partial (f+\kappa L\dist_{\Gamma}\circ F)(\bar x) = \nabla f(\bar x)+\nabla F(\bar x)^TN_{\Gamma}(F(\bar x)),
\]
which gives the M-stationarity at $\xb$.
\end{proof}
}\fi

%% file: S3-CQandMS_v2_1.tex
\section{New constraint qualifications for GMP}\label{sec:NewCQ}

\noindent
In this section we are primarily concerned with constraint qualifications for the general mathematical program  \eqref{eq:GMP}.
In particular, we investigate directional counterparts of pseudo- and quasi-normality introduced in \cite{BYZ19},
and introduce a new CQ called PQ-normality that unifies pseudo- and quasi-normality. We then show that each of these CQs implies MSCQ,
and hence recover statements (i) and (ii) of Proposition \ref{The : CQsToMS}.
Afterwards, we propose various sufficient conditions for these CQs under some additional structural assumptions.
Hence, when applied to the disjunctive constraints in Section \ref{sec:Disjunctive Sets},
these conditions also recover statements (iii) and (iv) of Proposition \ref{The : CQsToMS}.

\subsection{Directional constraint qualifications and PQ-normality}\label{sec:DirCQ}

\noindent
In \cite[Corollary 4.1]{BYZ19}, it was shown that metric subregularity is implied by directional quasi-normality, see Definition \ref{def:DirCQ}.
Here, we propose a different proof that follows the techniques used, e.g., in \cite[Lemma 4.4]{KaS 10} and \cite[Lemma 5.1.]{GuoYeZhang13}.
Note that our main tools are the Ekeland's variational principle and the rich subdifferential calculus for the distance function from Proposition \ref{prop:Dist},
and this approach is novel even is the nondirectional setting.
We start with the following observation, where we invoke definitions of $\Lambda^0(\xb;u)$ and $\Lambda^0(\xb)$.

\begin{lemma} \label{Lem : DirViolMSCQ}
Let $\bar x$ be feasible for \eqref{eq:GMP} such that MSCQ is violated at $\xb$.
Then there exist sequences $\{x^k\notin \cX\} \to \xb$ and
$\{\xi^k \in \p \left(\dist_\Gamma\circ F\right)(x^k)\}\to 0$ as well as
 $u\in \R^n\setminus\{0\}$ with $\|u\|=1$ such that
\begin{equation}\label{eq : DirVms}
\frac{x^k - \xb}{\norm{x^k - \xb}} \to u,\quad \frac{y^k - F(\xb)}{\norm{x^k - \xb}} \to \nabla F(\xb)u\quad (y^k\in P_\Gamma(F(x^k)))\AND  \nabla F(\xb) u \in T_{\Gamma}(F(\xb)).
\end{equation}
\end{lemma}
\begin{proof} Violation of MSCQ at $\bar x$ readily yields a sequence $\{\tilde x^k\} \to \xb$ with $\dist_\cX(\tilde x^k)>k \dist_\Gamma(F(\tilde x^k))$.
We put $\varepsilon_k:=\dist_{\Gamma}(F(\tilde x^k))$ and find that $\tilde x^k$ is an $\varepsilon_k$-minimizer of $\dist_\Gamma \circ F$ for all $k\in \N$.
Hence by Ekeland's variational principle \cite[Proposition 1.43]{RoW 98} with $\delta=\frac{1}{k}\;(k\in \N)$, there exists a sequence $\{x^k\}$ such that
$x^k=\argmin\left\{\dist_\Gamma\circ F+\frac{1}{k}\|(\cdot)- x^k\|\right\}$  and $\|x^k-\tilde x^k\|\leq k\varepsilon_k<\dist_{\cX}(\tilde x^k)$ for all $k\in \N$.
This implies $\{x^k\notin \cX\}\to \bar x$ as well as
$0\in \p(\dist_\Gamma\circ F)(x^k)+\frac{1}{k}\bB$ for all $k\in \N$
by applying a nonsmooth Fermat's rule (cf. \cite[Theorem 10.1]{RoW 98})
and invoking a sum rule for locally Lipschitz functions (cf. \cite[Exercise 10.10]{RoW 98}).
In particular, there exists a sequence $\{\xi^k\in \p(\dist_\Gamma\circ F)(x^k)\}\to 0$. As $x^k\neq \bar x$, w.l.o.g. we may assume that $\frac{x^k-\bar x}{\|x^k-\bar x\|}\to u$  with $\|u\|=1$. Now let $y^k\in P_{\Gamma}(F(x^k))$ for all $k\in \N$. Then
\begin{equation}\label{eq:AuxIneq}
\left\|\frac{y^k-F(\bar x)}{\|x^k-\bar x\|}-\nabla F(\bar x)u\right\|  \leq  \frac{\|y^k-F(x^k)\|}{\|x^k-\bar x\|}+\left\|\frac{F(x^k)-F(\bar x)}{\|x^k-\bar x\|}-\nabla F(\bar x)u\right\|\quad (k\in \N).
\end{equation}
As $x^k$ minimizes $\dist_\Gamma\circ F+\frac{1}{k}\|(\cdot)- x^k\|$ for all $k\in \N$, we find that
$\dist_{\Gamma}(F(x^k)) \leq 1/k \|\bar x-x^k\|$.
Hence we infer that the first term on the right in \eqref{eq:AuxIneq} satisfies
\[
\frac{\|y^k-F(x^k)\|}{\|x^k-\bar x\|}=\frac{\dist_{\Gamma} (F(x^k))}{\|x^k-\bar x\|}\leq \frac{1}{k}\to 0.
\]
The second term on the right in \eqref{eq:AuxIneq} goes to zero by differentiability of $F$ and we conclude from \eqref{eq:AuxIneq} that $\frac{y^k-F(\bar x)}{\|x^k-\bar x\|}\to\nabla F(\bar x)u$. Finally, as $y^k\in \Gamma$ for all $k\in \N$, we have $\nabla F(\bar x)u\in T_{\Gamma}(F(\bar x))$.
\end{proof}

\begin{theorem}\label{th:Nonexist}
 Let $\xb$ be feasible for \eqref{eq:GMP} and assume that the following holds: For every
 $ u\in \R^n$  with $\|u\|=1$ and  $\nabla F(\xb) u \in T_{\Gamma}(F(\xb))$
 there does \underline{not} exist a nonzero $\bar \lambda \in \Lambda^0(\xb;u)$
 that satisfies the following condition:
 There exists a sequence $\{(x^k,y^k,\lambda^k) \in \R^n\times \Gamma\times \R^d\}\to (\bar x,F(\bar x), \bar \lambda)$ such that  for all $k\in \N$ we have
 \begin{eqnarray*}
 (x^k - \xb)/\norm{x^k - \xb} \to u, && (y^k - F(\xb))/\norm{x^k - \xb} \to \nabla F(\xb) u,\\
 \lambda^k\in \widehat N_{\Gamma}(y^k), &  & \bar \lambda_i (F_i(x^k)-y_i^k)>0 \quad(\bar \lambda_{i}\neq 0).
 \end{eqnarray*}
  Then MSCQ is fulfilled at $\xb$.
\end{theorem}
\begin{proof} Assume that MSCQ is not satisfied at $\xb$. Consider
sequences $\{x^k\notin \cX\}\to \bar x$, $\{\xi^k\in \p \left(\dist_\Gamma\circ F\right)(x^k)\}\to 0$ and $u\in \R^n$  with $\|u\|=1$ provided by Lemma \ref{Lem : DirViolMSCQ}.
Recall that
\begin{equation*}
\partial (\dist_\Gamma\circ F)(x)\subset \nabla F(x)^T\partial \dist_\Gamma(F(x)) \quad (x\in \R^n),
\end{equation*}
see Proposition \ref{prop:Dist} (ii). Moreover, by Proposition \ref{prop:Dist} (i), it holds that
\begin{equation*}
\p \dist_\Gamma(F(x^k))=\frac{F(x^k)-P_\Gamma(F(x^k))}{\dist_\Gamma(F(x^k))}\quad (k\in \N),
\end{equation*}
since $x^k\notin \cX\;(k\in \N)$. Consequently, there exists $\{y^k\in P_\Gamma(F(x^k))\}$ such that with
\begin{equation}\label{eq:lambda_k}
\lambda^k:=\frac{F(x^k)-y^k}{\dist_{\Gamma}(F(x^k))}
\end{equation}
we have
\begin{equation}\label{eq:xi_k}
 \xi^k=\nabla F(x^k)^T \lambda^k \AND\|\lambda^k\|=1\quad  (k\in \N).
\end{equation}
Moreover, by the definition of $\lambda^k$ in \eqref{eq:lambda_k} and the fact that $y^k\in P_{\Gamma}(F(x^k))\, (k\in \N)$,
\cite[Example 6.16]{RoW 98} implies that
\begin{equation*}
\lambda^k\in \widehat N_\Gamma(y^k)\quad (k\in \N).
\end{equation*}
Since $\{\lambda^k\}$ is bounded, we may assume w.l.o.g. that
$\lambda^k\to \bar \lambda$ for some $\bar \lambda \neq 0 $.
Then from   \eqref{eq:lambda_k} we  infer that $y^k\to F(\bar x)$.
 Hence, passing to the limit in \eqref{eq:xi_k} we obtain
\begin{equation*}
0=\nabla F(\bar x)^T \bar \lambda \AND \bar \lambda \neq 0.
\end{equation*}
Now, if $\bar \lambda_i>0$ then w.l.o.g. $F_i(x^k)-y^k_i=\dist_\Gamma(F(x^k)) \lambda_i^k>0$ and hence $\bar \lambda_i(F_i(x^k)-y_i^k)>0$. Analogously, we argue for $\bar \lambda_i<0$. Altogether, we find that
\begin{equation*}
\bar \lambda_i(F_i(x^k)-y_i^k)>0\quad \text{if}\quad \bar \lambda_i\neq 0\quad (k\in \N).
\end{equation*}
Finally, Lemma \ref{Lem : DirViolMSCQ} yields that
$(y^k - F(\xb))/\norm{x^k - \xb} \to \nabla F(\xb) u$,
showing $\bar \lambda \in N_{\Gamma}(F(\xb),\nabla F(\xb)u)$, which establishes a contradiction.
\end{proof}

\noindent
Instead of directly extracting directional versions of quasi- and pseudo-normality from Theorem \ref{th:Nonexist}, we
introduce the notion of PQ-normality which serves as a bridge between pseudo- and quasi-normality,
which are then identified as the two extreme cases of PQ-normality. We strongly emphasize that introducing PQ-normality
does not merely serve the academic purpose of unifying
the two concepts. In fact, it has important consequences for the class of programs in Section \ref{sec:DisjProd}
where the set $\Gamma$ possesses an underlying product structure in addition to its disjunctive nature.

First, we introduce additional notation.
For $z \in \R^d$ we denote by $z_i\;(i \in I := \{1,\ldots,d\})$ its scalar components.
More generally, suppose that $\R^d$ is expressed via $l (\leq d)$ factors as $\R^{d_1} \times \ldots \times \R^{d_l}$
and introduce the $d$ multi-indices $\mli:= (d_1, \ldots, d_l) \in \N^l$
with $\vert \mli \vert := d_1 + \ldots + d_l = d$.
Note that there is a one-to-one correspondence between such multi-indices
and factorizations of $\R^d$.
The components of some $z \in \R^d$ we denote as $z_\nu$
for $\nu \in I_{\mli}$, where $I_{\mli}$ is some
(abstract) index set of $l$ elements.
Note that we do not identify $I_{\mli}$ with $\{1,\ldots,l\}$ in order
to avoid ambiguity of notation,
e.g., $z_1 \subset \R$ stands only for the first, scalar, component of $z$.
Moreover, we use a Greek letter to indicate the vector components $z_\nu$ of $z$
and a Latin letter to indicate the scalar components $z_i$.

Given a multi-index $\mli$ fix $\nu \in I_{\mli}$.
The component $z_\nu$, vector in general, can also be written via its
scalar components, i.e., there exists an index set, denoted by $I^{\nu}$,
such that $z_\nu = (z_i)_{i \in I^{\nu}}$.
Note that $\cup_{\nu \in I_{\mli}} I^{\nu} = I$.
Finally, given two multi-indices $\mli,\mli^{\prime}$
with $\vert \mli \vert = \vert \mli^{\prime} \vert = d$,
we say that $\mli^{\prime}$ is a {\em refinement} of
$\mli$ and write
$\mli^{\prime} \subset \mli$, provided
for every $\nu \in I_{\mli}$ there exists an index set
$I^{\nu}_{\mli^{\prime}}$ such that
\[z_\nu=(z_{\nu^{\prime}})_{\nu^{\prime} \in I^{\nu}_{\mli^{\prime}}}
 \ \textrm{ and } \ I_{\mli^{\prime}} = \cup_{\nu \in I_{\mli}} I^{\nu}_{\mli^{\prime}}.
\]
Note that the special multi-indices
$\mli^P:= d \in \N^1$ and $\mli^Q:= (1,\ldots,1) \in \N^{d}$
are in fact maximal and minimal in the sense that
for any multi-index $\mli \in \N^l$ with $\vert \mli \vert = d$
one has $\mli^Q \subset \mli \subset \mli^P$.

The following example illustrates the use of the above notation.

\begin{example}
Let $d=7$, $I:=\{1,\ldots,7\}$ and consider a multi-index $\mli:=(1,4,2)$ corresponding
to the factorization $\R^7 = \R \times \R^4 \times \R^2$.
Consider also an element $z=(z_1,\ldots,z_7) \in \R^7$.
Since $\delta$ has three components,  we may set, e.g.,
$I_{\mli} = \{a,b,c\}$ yielding
$z_a = z_1$, $z_b = (z_2,z_3,z_4,z_5)$ and $z_c = (z_6,z_7)$.
Clearly, we have $I^{a} = \{1\}$, $I^{b} = \{2,3,4,5\}$ and $I^{c} = \{6,7\}$.

Moreover, the multi-index $\mli^{\prime}:=(1,3,1,1,1)$
is a refinement of $\mli$, since we may set
\[I^{a}_{\mli^{\prime}} := \{a\}, \ I^{b}_{\mli^{\prime}}:= \{b_1,b_2\} \textrm{ and } I^{c}_{\mli^{\prime}}:= \{c_1,c_2\}\]
to obtain
\begin{eqnarray*}
&z_a = z_1, \ z_{b_1} = (z_2,z_3,z_4), \ z_{b_2}=z_5,
\ z_{c_1} = z_6, \ z_{c_2} = z_7,& \\
\textrm{and } & I_{\mli^{\prime}} = I^{a}_{\mli^{\prime}} \cup
I^{b}_{\mli^{\prime}} \cup I^{c}_{\mli^{\prime}}
=\{a,b_1,b_2,c_1,c_2\}, \quad
z_a = z_a, \ z_b = (z_{b_1},z_{b_2}), \ z_c = (z_{c_1},z_{c_2}).&
\end{eqnarray*}
\end{example}

We now proceed with the definition of PQ-normality which embeds quasi- and pseudo-normality as extremal cases in a whole family of constraint qualifications.

\begin{definition}[PQ-normality]\label{def:DirCQ} Let $\bar x\in\cX$ be feasible for \eqref{eq:GMP}, consider $u \in \R^n$ with $\|u\|=1$,  and let  $\mli \in \N^{l}$ be a multi-index such that $|\delta|=d$.
We say that
\begin{itemize}
\item[(i)] {\em PQ-normality w.r.t. $\mli$} holds at $\bar x$, if \underline{there is no nonzero}
$\bar \lambda \in \Lambda^0(\xb)$ that satisfies the following condition:
 There exists a sequence $\{(x^k,y^k,\lambda^k) \in \R^n\times \Gamma\times \R^d\}\to (\bar x,F(\bar x), \bar \lambda)$ with $\lambda^k\in \widehat N_{\Gamma}(y^k)$
 and
\begin{equation} \label{eqn : DirPQN2}
 \ip{\bar \lambda_{\nu}}{F_{\nu}(x^k)-y_{\nu}^k} >0 \quad \
 (\nu \in I_{\mli}(\bar\lambda) := \{\nu \in I_{\mli} \mv
 \bar \lambda_{\nu} \neq 0\}, \, k\in\N).
 \end{equation}
\item[(ii)] {\em PQ-normality w.r.t. $\mli$ in direction $u$} holds at $\bar x$, if \underline{there is no nonzero} $\bar \lambda \in \Lambda^0(\xb;u)$ that satisfies the following condition:
 There exists a sequence $\{(x^k,y^k,\lambda^k) \in \R^n\times \Gamma\times \R^d\}\to (\bar x,F(\bar x), \bar \lambda)$ with $\lambda^k\in \widehat N_{\Gamma}(y^k)$, \eqref{eqn : DirPQN2} and
 \begin{equation} \label{eqn : DirPN1}
 (x^k - \xb)/\norm{x^k - \xb} \to u, \quad (y^k - F(\xb))/\norm{x^k - \xb} \to \nabla F(\xb) u.
 \end{equation}
\end{itemize}
We say that {\em directional PQ-normality w.r.t. $\mli$} holds at $\bar x$, if
PQ-normality w.r.t. $\mli$ in direction $u$ holds at $\xb$ for all $u \in \R^n$ with $\|u\|=1$.
In particular, we refer to PQ-normality w.r.t. $\mli^P$ (in direction $u$)
as pseudo-normality (in direction $u$), while PQ-normality w.r.t. $\mli^Q$
we call quasi-normality.
\end{definition}

It is clear from the definition that PQ-normality w.r.t. $\mli$ implies PQ-normality w.r.t. $\mli^{\prime}$
provided $\mli^{\prime} \subset \mli$.
In particular, since $\mli^Q \subset \mli \subset \mli^P$ for all
$\mli \in \N^l$ with $\vert \mli \vert = d$,
we conclude that pseudo-normality implies
PQ-normality w.r.t. any $\mli$ and this further implies quasi-normality.
Naturally, all of the above comments remain true for the corresponding directional CQs.

For the sake of completeness, we reformulate Theorem \ref{th:Nonexist} in terms of directional PQ-normality.

\begin{theorem} \label{The : MainDQNtoMS}
 Let $\xb$ be feasible for \eqref{eq:GMP} and let the directional PQ-normality w.r.t. any $\mli \in \N^l$,
 in particular directional pseudo- or quasi-normality, hold at $\xb$.
 Then MSCQ is fulfilled at $\xb$.
\end{theorem}

We point out that directional quasi-normality is strictly weaker than both
FOSCMS (clear from the definition of the respective CQs) as well as quasi-normality, see Example \ref{ex: DirCQ}.
Hence it constitutes, to the best of our knowledge, one of the weakest conditions to imply MSCQ for the general optimization problem \eqref{eq:GMP},
which can still be efficiently verified in some very important cases as shown in Section \ref{sec:SimpStan} and Section \ref{sec:SimpDir} below.

The same directional versions of pseudo- and quasi-normality were independently introduced in a recent paper by Bai et al. \cite{BYZ19}.
In order to show that these imply MSCQ, Bai et al. build on  results from \cite[Corollary 1, Remarks 1 and 2]{Gfr14b}.
Hence, we believe our alternative proof can provide additional insight on the role of pseudo- and quasi-normality in verifying  MSCQ.
More importantly, in what follows, we focus on simplifying these conditions  under specific structural assumptions on the feasible set,
which is a crucial step to facilitate their use.

\subsection{Simplified CQs and second-order sufficient conditions: The standard case}\label{sec:SimpStan}

For some important instances of the general program \eqref{eq:GMP}, the concepts of pseudo- and quasi-normality were introduced
without the undesirable additional sequence $\{y^k\}$, see \cite{BeO 02} for standard NLPs and \cite{KaS 10} for MPCCs.
In the remaining part of this section, we address the question as to when this is possible for more general instances of \eqref{eq:GMP},
working with the generalized notion of PQ-normality.
In turn, in the remainder of this section, $\mli$
denotes a multi-index in $\N^{l}$ for some $l \in \{1,\ldots,d\}$
with $\vert \mli \vert = d$ unless stated otherwise.
As a result of dropping the sequence $\{y^k\}$, we obtain a characterization of PQ-normality via  an extremal condition,
which, in turn, yields several sufficient conditions for PQ-normality.

For clarity of exposition, we split our analysis into the standard (non-directional) and the directional case.

We begin our study of the non-directional case by the following straightforward result, which follows readily from
definition of PQ-normality using the sequences $y^k:=F(\bar x)$
and $\lambda^k := \bar \lambda$,
taking also into account Lemma \ref{lem:Aux}
and the arguments in the proof of Corollary \ref{cor:Aux}.

\begin{lemma} \label{Pro:ModCQ}
 Let $\bar x $ be feasible for \eqref{eq:GMP}.
 If PQ-normality w.r.t. $\mli$ holds at $\bar x$ then there is no nonzero
 $\bar \lambda \in \Lambda^0(\xb)$ that satisfies the following condition:
 There exists a sequence $\{x^k\} \to \bar x$ with
\begin{equation} \label{eq : MPQNdef}
\ip{\bar \lambda_{\nu}}{F_{\nu}(x^k)-F_{\nu}(\xb)} >0 \quad \ (\nu \in I_{\mli}(\bar \lambda), \;k\in K).
\end{equation}
\end{lemma}

\noindent
Note that in case of MPCCs, by the geometry of the feasible set and the resulting normal cones, one always has $\ip{\bar \lambda}{F(\bx)} = 0$.
Thus the conditions used in \cite{KaS 10}
simplify to $\ip{\bar \lambda}{F(x^k)}>0$ and $\bar \lambda_i F_i(x^k) > 0$ if $\bar \lambda_i \neq 0$, respectively.
However, in the general setting of problem \eqref{eq:GMP}, as well as
in the case of general disjunctive constraints,
we cannot make this simplification.
In order to obtain the reverse implication, however, we have to impose some additional assumptions on the constraints of \eqref{eq:GMP}.

\begin{assumption} \label{Ass} Let $\mli$ be a multi-index and let $\bar x$ be feasible for \eqref{eq:GMP}.
Assume that for every $\bar \lambda \in \Lambda^0(\xb)$
 and every sequence $\{(y^k,\lambda^k) \in \Gamma\times \R^d\}\to (F(\xb), \bar \lambda)$
 with $\lambda^k\in \widehat N_{\Gamma}(y^k)$, there exists a subsequence $K \subset \N$ such that
  \begin{equation}\label{eq:SuffCond2}
\ip{\bar \lambda_{\nu}}{y_{\nu}^k - F_{\nu}(\xb)} \geq 0 \quad \ (\nu \in I_{\mli}(\bar \lambda), \;k\in K).
  \end{equation}
\end{assumption}

\begin{theorem}[Simplified PQ-normality under Ass. \ref{Ass}] \label{The:ModCQ}
 Let $\bar x$ be feasible for \eqref{eq:GMP} and $\mli$ such that Assumption \ref{Ass} holds.
 Then PQ-normality w.r.t. $\mli$ at $\xb$ is equivalent to the following {\em simplified PQ-normality} w.r.t. $\mli$ at $\xb$, i.e.
\begin{center}
There is no nonzero $\bar \lambda \in \Lambda^0(\xb)$
   such that there exists a sequence $\{x^k\} \to \bar x$
   fulfilling \eqref{eq : MPQNdef}.
      \end{center}
\end{theorem}
\begin{proof}
 The fact that PQ-normality implies the simplified PQ-normality follows from Lemma \ref{Pro:ModCQ}.

 In turn, if PQ-normality w.r.t. $\mli$ is violated, there exist $\bar \lambda \in \Lambda^0(\xb)\setminus\{0\}$ and
 $\{(x^k,y^k,\lambda^k) \in \R^n\times \Gamma\times \R^d\}\to (\bar x,F(\bar x),\bar \lambda)$ with $\lambda^k\in \widehat N_{\Gamma}(y^k)$ and
 $\ip{\bar \lambda_{\nu}}{F_{\nu}(x^k)-y_{\nu}^k} >0$
 for all $\nu \in I_{\mli}(\bar\lambda)$. Relabeling $\{x^k\}$ by only using the indices $k\in K$ and then summing up the above expression with \eqref{eq:SuffCond2} for all $k\in K$ shows that the simplified PQ-normality is then violated as well.
\end{proof}

As the above theorem shows, under Assumption \ref{Ass}, the simplified PQ-normality is equivalent to PQ-normality, hence sufficient for MSCQ.
Without Assumption \ref{Ass} this is, in general, false, see Example \ref{ex:MSviol}.
In the following sections, however, we deal with various types of optimization problems which automatically satisfy Assumption \ref{Ass} for a suitable multi-index,
including $\mli^P$ and $\mli^Q$, at every feasible point.

As we will now show, Theorem \ref{The:ModCQ} also reveals an interesting connection between PQ-normality and vector optimization.
This, in turn, paves the way to a variety of sufficient conditions for PQ-normality, hence also for MSCQ.

Let us recall some standard terminology from multiobjective  optimization \cite{Ehr 05,Jah 11}. Given  $\varphi : \R^n \to \R^q$, a point $\xb$ is called a
{\em local weak efficient solution} of the unconstrained vector optimization problem
$\max_{x \in \R^n} \varphi(x)$ if there exists a neighborhood $U$ of $\xb$
such that no $x \in U$ satisfies $\varphi_j(x) > \varphi_j(\xb)$ for all $j = 1,\ldots,q$.
Given $\mli=(d_1, \ldots, d_l) \in \N^{l}$ and $\lambda=(\lambda_{\nu})_{\nu \in I_{\mli}} \in \R^{d_1} \times \ldots \times \R^{d_l} = \R^d$,
we define the function
\begin{equation}\label{eq:varphi}
\varphi^{\lambda}:\R^n\to \R^{|I_{\mli}(\lambda)|},\quad \varphi^{\lambda}(x):=(\ip{\lambda_{\nu}}{F_{\nu}}(x))_{\nu\in I_{\mli}(\lambda)}.
\end{equation}

\begin{theorem} \label{Pro : PQNormVSMax}
Let $\bar x$ be feasible for \eqref{eq:GMP} and let Assumption \ref{Ass}
for some $\mli$ be fulfilled.
Then PQ-normality w.r.t. $\mli$ holds at $\bx$ if and only if for every $\bar \lambda \in \Lambda^0(\xb)$, the vector $\xb$ is a local weak efficient solution of the unconstrained vector optimization problem
$ \max_{x \in \R^n} \varphi^{\bar \lambda}(x)$
 for $\varphi^{\bar \lambda}$ given by \eqref{eq:varphi}.
\end{theorem}
\begin{proof}
If there exists $\bar \lambda\in \Lambda^0(\bar x)$ such that $\bar x$ is not a local weak efficient solution of $\max_{x \in \R^n} \varphi^{\bar \lambda}(x)$, then $\bar \lambda\neq 0$ and there exists $\{x^k\} \to \bar x$ such that
$\ip{\bar \lambda_{\nu}}{F_{\nu}(x^k)} > \ip{\bar \lambda_{\nu}}{F_{\nu}(\xb)}$ for all $\nu \in I_{\mli}(\bar \lambda)$ and all $k\in \N$. This shows that PQ-normality w.r.t. $\mli$ is violated due to Theorem \ref{The:ModCQ}.

In turn, if PQ-normality w.r.t. $\mli$ is violated, there exists $\bar \lambda \in \Lambda^0(\bar x)\setminus \{0\}$ and a sequence $\{x^k\} \to \bar x$ such that
$\ip{\bar \lambda_{\nu}}{F_{\nu}(x^k) - F_{\nu}(\xb)}>0$ for all $\nu \in I_{\mli}(\bar \lambda)$ and all $k\in \N$, which shows that $\bar x$ is not a local weak efficient solution of $\max_{x \in \R^n} \varphi^{\bar \lambda}(x)$.
\end{proof}

This simple observation has some significant consequences.
In particular, it allows us to use the standard sufficient conditions
for a local weak efficient solution to obtain the following point-based sufficient condition for PQ-normality.
\begin{corollary}[Sufficient condition for PQ-normality] \label{Cor : PQNormSC}
Let $\bar x$ be feasible for \eqref{eq:GMP} with $F$ twice differentiable at $\xb$
and let Assumption \ref{Ass} for some $\mli$ be fulfilled.
Then PQ-normality w.r.t. $\mli$, in particular MSCQ, holds at $\bx$
under the following condition:
For every $\bar\lambda \in \Lambda^0(\xb)\setminus\{0\}$, every $u \in \R^n\setminus\{0\}$ with $\ip{\bar\lambda_{\nu}}{\nabla F_{\nu}(\xb) u} = 0$ for all $\nu \in I_{\mli}(\bar\lambda)$ and every $w$ with $\ip{w}{u} = 0$ one has
\begin{equation} \label{eq : SSOSCMPQN}
\min_{\nu \in I_{\mli}(\bar\lambda)} \left( \ip{\bar\lambda_{\nu}}{\nabla F_{\nu}(\xb) w} + u^T \nabla^2 \ip{\bar\lambda_{\nu}}{F_{\nu}}(\xb) u \right) < 0.
\end{equation}
\end{corollary}
\begin{proof} Consider $\bar\lambda \in \Lambda^0(\xb)\setminus\{0\}$
and $\varphi^{\bar\lambda}$ given by \eqref{eq:varphi} and let $z \in \R^n$ be arbitrary. Then
 \begin{equation} \label{eq : LinL0}
  \sum_{\nu \in I_{\mli}(\bar\lambda)} \nabla \varphi^{\bar\lambda}_\nu(\xb) z =
 \sum_{\nu \in I_{\mli}(\bar\lambda)} \ip{\bar\lambda_{\nu}}{\nabla F_{\nu}(\xb) z} = \ip{\bar\lambda}{\nabla F(\xb) z} = 0,
 \end{equation}
 since $\bar\lambda \in \Lambda^0(\xb)$. Hence, every $u$ with $\nabla \varphi^{\bar\lambda}_\nu(\xb) u \geq 0$
 for all $\nu \in I_{\mli}(\bar\lambda)$ in fact fulfills $\nabla \varphi^{\bar\lambda}_\nu(\xb) u = \ip{\bar\lambda_{\nu}}{\nabla F_{\nu}(\xb) u} = 0$ for all $\nu \in I_{\mli}(\bar\lambda)$.
 The result thus follows from \cite[Theorem 4]{Bi06} and Theorem \ref{Pro : PQNormVSMax}.
\end{proof}

\begin{remark}
Note that $\bar\lambda \in \Lambda^0(\xb)$
 implies the first-order necessary conditions for local efficient solution,
 $\min_{\nu \in I_{\mli}(\bar\lambda)}
 \ip{\bar\lambda_{\nu}}{\nabla F_{\nu}(\xb) w} \leq 0$ for all $w \in \R^n$,
 as can be seen from \eqref{eq : LinL0}.
\end{remark}

The above corollary motivates the following definition.
\begin{definition}
Given a feasible point $\bar x$ for \eqref{eq:GMP} and a
multi-index $\mli$,
we say that the second-order sufficient condition for PQ-normality
w.r.t. $\mli$, SOSCPQN($\mli$) for short,
holds at $\xb$ provided:
For every $\bar\lambda \in \Lambda^0(\xb)\setminus\{0\}$, every $u \in \R^n\setminus\{0\}$ with $\ip{\bar\lambda_{\nu}}{\nabla F_{\nu}(\xb) u} = 0$ for all $\nu \in I_{\mli}(\bar\lambda)$ and every $w$ with $\ip{w}{u} = 0$ one has \eqref{eq : SSOSCMPQN}.

Moreover, we refer to SOSCPQN($\mli^P$) and SOSCPQN($\mli^Q$)
as second-order sufficient condition for pseudo-/quasi-normality
(SOSCPN and SOSCQN), respectively.
\end{definition}

Naturally, one can also consider higher-order sufficient conditions. We do so in Section 4,
where we focus on pseudo-normality.
Note that pseudo-normality is connected
to standard maximality since $\varphi^{\lambda}$ is a scalar function in that case.

The following example shows that SOSCPN on its own, i.e., without Assumption \ref{Ass} for $\mli^P$, does not guarantee pseudo-normality, not even MSCQ.

\begin{example} \label{ex:MSviol}
Consider $\Gamma \subset \R^2$ given by $\Gamma:=\{y \in \R^2 \mv y_2 \geq \vert y_1 \vert^{3/2}\}$
and $F: \R \to \R^2$ defined by $F(x) := (x,x^2)^T$ and let $\xb := 0$.
Clearly $\nabla F(\xb)=(1,0)^T$ and $\Lambda^0(\xb) = \R_+ (0,-1)^T.$
Thus, for every
$\lambda \in \Lambda^0(\xb) \setminus \{0\}$
and every $u \in \R\setminus \{0\}$ we have
$u^T \nabla^2 \ip{\lambda}{F}(\xb) u = -2 \alpha u^2 < 0$,
where $\alpha>0$ is such that $\lambda = (0,-\alpha)$,
showing that SOSCPN holds at $\xb$.
On the other hand, for a sequence $\{x_k\} \to 0$ we obtain $\dist_{F^{-1}(\Gamma)}(x_k) = \vert x_k \vert$,
while
\[\dist_\Gamma(F(x_k)) \leq \norm{(x_k,x_k^2) - (x_k,\vert x_k \vert^{3/2})} \leq \vert x_k \vert^{3/2},\]
showing the violation of MSCQ and consequently of pseudo-normality as well.
\end{example}

\noindent
We point out that the set $\Gamma$ in Example \ref{ex:MSviol} equals $\epi |\cdot|^{3/2}$ and is therefore convex, yet SOSCPN still does not imply MSCQ.
\begin{theorem} \label{Pro:SOSCMPNtoSOSCMPQN}
 Let $\bar x$ be feasible for \eqref{eq:GMP} with $F$ twice differentiable at $\xb$ and consider two multi-indices $\mli \in \N^l, \mli^{\prime} \in \N^{l^{\prime}}$
 with $\mli^{\prime} \subset \mli$.
 Then SOSCPQN($\mli$) implies SOSCPQN($\mli^{\prime}$).
 In particular, we have
 SOSCPN $\Rightarrow$ SOSCPQN($\mli$) $\Rightarrow$ SOSCQN.
\end{theorem}
\begin{proof}
 Consider $0 \neq \bar\lambda \in \Lambda^0(\xb)$, $0 \neq u \in \R^n$ and $w \in \R^n$ with
 $\ip{\bar\lambda_{\nu^{\prime}}}{\nabla F_{\nu^{\prime}}(\xb) u} = 0$
 for all $\nu^{\prime} \in I_{\mli^{\prime}}(\bar\lambda)$ and $\ip{w}{u} = 0$.
 For any $\nu \in I_{\mli}(\bar\lambda)$, we find some
 index set $I^\nu_{\mli^{\prime}} \subset I_{\mli^{\prime}}$
 such that $z_{\nu} = (z_{\nu^{\prime}})_{\nu^{\prime} \in I^\nu_{\mli^{\prime}}}$
 by $\mli^{\prime} \subset \mli$.
 Summing up $\ip{\bar\lambda_{\nu^{\prime}}}{\nabla F_{\nu^{\prime}}(\xb) u} = 0$
 over $I^\nu_{\mli^{\prime}}$ yields
 $\ip{\bar\lambda_{\nu}}{\nabla F_{\nu}(\xb) u} = 0$.
 Thus, we can apply SOSCPQN($\mli$) in order to infer the existence of
 $\bar\nu \in I_{\mli}(\bar\lambda)$ such that
 \if{
 \[\ip{\bar\lambda_{\bar\nu}}{\nabla F_{\bar\nu}(\xb) w} + u^T \nabla^2 \ip{\bar\lambda_{\bar\nu}}{F_{\bar\nu}}(\xb) u < 0.\]
 Again, $\mli^{\prime} \subset \mli$,
 yields the existence of index set $I^{\bar\nu}_{\mli^{\prime}} \subset I_{\mli^{\prime}}$
 such that $z_{\bar\nu} = (z_{\bar\nu^{\prime}})_{\bar\nu^{\prime} \in I^{\bar\nu}_{\mli^{\prime}}}$ and we obtain
 }\fi
 \begin{equation*}
  \sum_{\bar\nu^{\prime} \in I^{\bar\nu}_{\mli^{\prime}}} \left( \ip{\bar\lambda_{\bar\nu^{\prime}}}{\nabla F_{\bar\nu^{\prime}}(\xb) w} + u^T \nabla^2 \ip{\bar\lambda_{\bar\nu^{\prime}}}{F_{\bar\nu^{\prime}}}(\xb) u \right) =
  \ip{\bar\lambda_{\bar\nu}}{\nabla F_{\bar\nu}(\xb) w} + u^T \nabla^2 \ip{\bar\lambda_{\bar\nu}}{F_{\bar\nu}}(\xb) u < 0.
 \end{equation*}
 This yields, however, that SOSCPQN($\mli^{\prime}$) is fulfilled.

 The second statement now follows from the obvious relation
 $\mli^{Q} \subset \mli \subset \mli^{P}$ valid for any $\mli$.
\end{proof}
\noindent
The above theorem holds regardless of Assumption \ref{Ass}.
If one seeks to use any of the sufficient conditions
to get metric subregularity, however,
one clearly needs it, see Examples
\ref{ex:MSviol} and \ref{ex:MSviolUnderMQN}.
Note also that if Assumption \ref{Ass} holds for $\mli^{\prime}$,
then it also holds for any $\mli \supset \mli^{\prime}$.

The following example shows that SOSCQN is, in fact, strictly milder than SOSCPN.
Moreover, it demonstrates that one can effectively verify MSCQ by means of SOSCQN
even when pseudo-normality is not fulfilled.

\begin{example} \label{ex:PQNyPNn}
Let $\Gamma := \Gamma_1 \times \Gamma_2 \subset \R^2$
for two convex polyhedral sets $\Gamma_1 = \Gamma_2:= \R_-$
and let $F:=(F_1,F_2)^T: \R \to \R^2$ for
$F_1(x) := -x$ and $F_2(x) := x + x^2$ and let $\xb := 0$.
In particular, Assumption \ref{Ass} for $\mli^Q$ is fulfilled by Corollary \ref{Pro : QusinormEquiv}.
Clearly, $\nabla F_1(\xb)=-1$, $\nabla F_2(\xb)=1$ and hence
$\Lambda^0(\bar x) = \R_+ (1,1)^T.$

SOSCQN is fulfilled since for any $\lambda = (\lambda_1,\lambda_2) = \alpha(1,1)^T$ for some $\alpha > 0$ and for $u = \pm 1$
one has $\vert \lambda_i \nabla F_i(\xb) u \vert = \alpha \neq 0$, $i=1,2$.
In particular, quasi-normality and MSCQ follows.

On the other hand, let $\bar \lambda:= (1,1)^T$ and consider a sequence $\{x_k\} \downarrow 0$. We obtain
\[\ip{\bar \lambda}{F(x_k) - F(\bx)}=-x_k + x_k + x_k^2 >0,\]
showing the violation of pseudo-normality.
\end{example}

The next example shows that, without Assumption \ref{Ass} for $\mli^Q$,
the simplified form of quasi-normality from Lemma \ref{Pro:ModCQ}
does not imply MSCQ even if $\Gamma$ is a convex polyhedral set.

\begin{example} \label{ex:MSviolUnderMQN}
Let $\Gamma \subset \R^2$ be convex polyhedral set given by
$\Gamma:=\{y \in \R^2 \mv y_2 \geq y_1 \}$
and $F: \R \to \R^2$ given by $F(x) := (x,\sin x)^T$ and let $\xb := 0$.
Clearly $\nabla F(\xb)=(1,1)^T$ and we find that $\Lambda^0(\bar x) = \R_+ (1,-1)^T.$
For every $\lambda = (\lambda_1,\lambda_2) = \alpha(1,-1)^T$ for some $\alpha > 0$
and every $x \in \R$ close to $\xb$ we have
$\lambda_1 (F_1(x)-F_1(\xb)) = \alpha x < 0$ if $x<0$ and
$\lambda_2 (F_2(x)-F_1(\xb)) = - \alpha \sin x \leq 0$ if $x\geq0$, showing that
the simplified form of quasi-normality holds at $\xb$.
On the other hand, for a sequence $\{x_k\} \downarrow 0$ we obtain $\dist_{F^{-1}(\Gamma)}(x_k) = \vert x_k \vert$,
while
\[\dist_\Gamma(F(x_k)) \leq \norm{(x_k,\sin x_k) - (x_k, x_k )} = o(\vert x_k \vert),\]
showing the violation of MSCQ.
\end{example}

\subsection{Simplified CQs and second-order sufficient conditions: The directional case}\label{sec:SimpDir}

In this subsection, we consider the directional case, where the situation is slightly different.

\begin{theorem} \label{The:DirModCQ}
 Let $\bar x$ be feasible for \eqref{eq:GMP} and consider $u \in \R^n$ with $\|u\|=1$. Then
under Assumption \ref{Ass} for $\mli$, PQ-normality w.r.t. $\mli$ at $\xb$ in direction $u$
holds if:
   there is no nonzero $\bar \lambda \in \Lambda^0(\xb;u)$
   such that there exists a sequence $\{x^k\} \to \bar x$
   with $(x^k - \xb)/\norm{x^k - \xb} \to u$ fulfilling \eqref{eq : MPQNdef}.
\end{theorem}
\begin{proof}
 The proof follows by the same arguments as used in the proof of Theorem
 \ref{The:ModCQ}.
\end{proof}

In contrast to the standard case, the following example shows
that the reverse implication in the above theorem is not true in general.
\begin{example} \label{Ex:DirVerDiff}
 Consider $\Gamma \subset \R^2$ given by $\Gamma:=\{y \in \R^2 \mv y_2 \leq y_1^{2}\}$
and $F: \R \to \R^2$ defined by $F(x) := (x,x^4)^T$ and let $\xb := 0$ and $u:= 1$.
Clearly $\nabla F(\xb)=(1,0)^T$ and $\Lambda^0(\xb;1)=\Lambda^0(\xb) = \R_+ (0,1)^T.$
Set $\bar \lambda := (0,1)^T$ and note that any sequence $\{x_k\} \downarrow 0$ fulfills
$(x_k - \xb)/\norm{x_k - \xb} \to u$ as well as \eqref{eq : MPQNdef}
for $\mli^P$, since $\ip{\bar \lambda}{F(x_k) - F(\xb)} = x_k^4 > 0$.

On the other hand, for arbitrary sequence $y^k=(y^k_1,y^k_2)^T \to F(\xb)=(0,0)^T$ with
$N_{\Gamma}(y^k) \neq \{0\}$ we have $y^k=(y_1^k,(y_1^k)^2)^T$.
Hence, for any $\lambda \in \R_+ (0,1)^T$ one has
$\ip{\lambda}{y^k - F(\bx)} = \lambda_2 (y_1^k)^2 \geq 0$,
showing that Assumption \ref{Ass} for $\mli^P$ is fulfilled.
Moreover $(y^k_1/x_k,(y^k_1)^2/x_k)^T = (y^k - F(\xb))/\norm{x_k - \xb} \to \nabla F(\xb) u = (1,0)^T$
yields $y^k_1/x_k \to 1$.
Then, however, we obtain
\[\ip{\lambda}{F(x_k) - y^k}= \lambda_2(x_k^4 - (y_1^k)^2) =
\lambda_2 x_k^2(x_k^2 - (y^k_1)^2/x_k^2) \leq 0,\]
showing that pseudo-normality at $\xb$ in direction $u$ is fulfilled.
\end{example}

Nevertheless, the previous theorem still allows us to use sufficient conditions.
Consider the following second-order sufficient condition for directional PQ-normality w.r.t. $\mli$, SOSCdirPQN($\mli$) for short.
\begin{proposition}[SOSCdirPQN($\mli$)]\label{Cor: SOSCdirPQN}
Let $\bar x$ be feasible for \eqref{eq:GMP} with $F$ twice differentiable at $\xb$
and let Assumption \ref{Ass} for some $\mli$ be fulfilled.
Then directional PQ-normality w.r.t. $\mli$, in particular MSCQ, holds at $\bx$
if the following SOSCdirPQN($\mli$) is fulfilled:
For every $u \in \R^n$ with $\|u\|=1$, every $0 \neq \bar\lambda \in \Lambda^0(\xb;u)$
with $\ip{\bar\lambda_{\nu}}{\nabla F_{\nu}(\xb) u} = 0$, for all $\nu \in I_{\mli}(\bar\lambda)$ and every $w$ with $\ip{w}{u} = 0$ condition \eqref{eq : SSOSCMPQN} is fulfilled.
\end{proposition}
\begin{proof}
 Assume that directional PQ-normality w.r.t. $\mli$ is violated.
 Theorem \ref{The:DirModCQ} yields the existence of
 $u \in \R^n$, $0 \neq \bar \lambda \in \Lambda^0(\xb;u)$
 and a sequence $\{x^k\} \to \bar x$ with $(x^k - \xb)/\norm{x^k - \xb} \to u$
 such that
 $\varphi^{\bar\lambda}_\nu(x^k) - \varphi^{\bar\lambda}_\nu(\xb) > 0$ for all $\nu \in I_{\mli}(\bar\lambda)$ with $\varphi^{\bar\lambda}$ as in \eqref{eq:varphi}.
 Hence, by passing to a subsequence if necessary, we can assume that $(\varphi(x^k) - \varphi(\xb))/\norm{\varphi(x^k) - \varphi(\xb)} \to p$ with $p \geq 0$ and $\norm{p} = 1$,
 where for simplification we dropped the upper index $\bar\lambda$ from $\varphi$.

 By Taylor expansion, we have
 \begin{equation} \label{eq : TaylorExp}
  \frac{\norm{\varphi(x^k) - \varphi(\xb)}}{\norm{x^k - \xb}^2} \frac{\varphi(x^k) - \varphi(\xb)}{\norm{\varphi(x^k) - \varphi(\xb)}} =
  \nabla \varphi(\xb) \frac{(x^k - \xb)}{\norm{x^k - \xb}^2} + u^T
  \nabla^2 \varphi(\xb) u + o(1),
 \end{equation}
 where $u^T
  \nabla^2 \varphi(\xb) u$ denotes the vector in $\R^{\vert I_{\mli}(\bar\lambda) \vert}$ with components
  $u^T
  \nabla^2 \varphi^{\bar\lambda}_\nu(\xb) u$ for $\nu \in I_{\mli}(\bar\lambda)$.
 If there exists a subsequence $K$ such that $\norm{\varphi(x^k) - \varphi(\xb)}/\norm{x^k - \xb}^2 \to \infty$,
 we conclude from \eqref{eq : TaylorExp} that
 \begin{equation*}
  \frac{\varphi(x^k) - \varphi(\xb)}{\norm{\varphi(x^k) - \varphi(\xb)}} = \nabla \varphi(\xb) \frac{(x^k - \xb)}{\norm{\varphi(x^k) - \varphi(\xb)}} + q^k,
 \end{equation*}
 where $q^k \to 0$ for $k \in K$. Passing to a subsequence if necessary, and taking into account
 that $\nabla \varphi(\xb) (x^k - \xb)/\norm{\varphi(x^k) - \varphi(\xb)} \in \Range(\nabla \varphi(\xb))$ with $\Range(\nabla \varphi(\xb))$
 being a closed set, we conclude that $p \in \Range(\nabla \varphi(\xb))$,
 i.e., there exists $z \in \R^n$ with
 $p = \nabla \varphi(\xb) z = (\ip{\bar\lambda_{\nu}}{\nabla F_{\nu}(\xb) z})_{\nu \in I_{\mli}(\bar\lambda)}$.
 This is, however, a contradiction with $\norm{p} = 1$,
 since we obtain that $p=0$ by $p \geq 0$ and \eqref{eq : LinL0},
 which clearly holds due to $\bar\lambda \in \Lambda^0(\xb;u) \subset \Lambda^0(\xb)$.

 Consequently, $\norm{\varphi(x^k) - \varphi(\xb)}/\norm{x^k - \xb}^2$ remains bounded
 and by passing to a subsequence $K$ if necessary we assume that $\norm{\varphi(x^k) - \varphi(\xb)}/\norm{x^k - \xb}^2 \to \alpha \geq 0$.
 Note also that in this case we get $\nabla \varphi(\xb) u = 0$
 by \eqref{eq : TaylorExp}.
 By similar arguments as before, \eqref{eq : TaylorExp} now yields the existence of $w$ such that
 \begin{equation*}
  \alpha p = \nabla \varphi(\xb) w + u^T \nabla^2 \varphi(\xb) u.
 \end{equation*}
 Moreover, we can clearly take $w$ with $\ip{w}{u} = 0$ since $\R^n$ is the direct sum of the span of $u$ and its orthogonal complement
 and $\nabla \varphi(\xb) u = 0$.
 The assumed SSOSCdirPQN($\mli$) \eqref{eq : SSOSCMPQN} implies the existence of $\nu \in I_{\mli}(\bar\lambda)$ with $\alpha p_\nu < 0$,
 a contradiction. This completes the proof.
\end{proof}
\begin{remark}
 Note that $\Lambda^0(\xb;u) \neq \emptyset$ includes the condition $\nabla F(\xb) u \in T_{\Gamma}(F(\xb))$.
\end{remark}

\if{
\noindent
In the definition of SOSCdirPQN($\mli$) we explicitly assume $\ip{\bar\lambda_{\nu}}{\nabla F_{\nu}(\xb) u} = 0$ for all $\nu \in I_{\mli}(\bar\lambda)$
in order to make it clear that SOSCdirPQN($\mli$) is indeed milder than SOSCPQN($\mli$).
In fact, we can omit it from the assumption since it actually follows from
$\bar\lambda \in \Lambda^0(\xb;u)$.
}\fi

As before, we will refer to SOSCdirPQN($\mli^P$)
and SOSCdirPQN($\mli^Q$) as
second-order sufficient condition for directional pseudo/quasi-normality (SOSCdirPN and SOSCdirQN).

The following directional counterpart of Theorem \ref{Pro:SOSCMPNtoSOSCMPQN}
follows by the same arguments.
\begin{theorem} \label{Pro:SOSCMPNtoSOSCMPQNdir}
 Let $\bar x$ be feasible for \eqref{eq:GMP} with $F$ twice differentiable at $\xb$ and consider two multi-indices $\mli \in \N^l, \mli^{\prime} \in \N^{l^{\prime}}$
 with $\mli^{\prime} \subset \mli$.
 Then SOSCdirPQN($\mli$) implies SOSCdirPQN($\mli^{\prime}$).
 In particular, we have
 SOSCdirPN $\Rightarrow$ SOSCdirPQN($\mli$) $\Rightarrow$ SOSCdirQN.
\end{theorem}

We point out here that, unlike in the non-directional case, we could not find an example to show that the above implications can be indeed strict,
so this remains an open question.

\subsection{Summary}

\noindent
We now summarize our findings of this section. We studied the directional versions of  pseudo- and quasi-normality, first established in  the paper by Bai et al.  \cite{BYZ19}.
In addition, we introduced the new concept of PQ-normality, together with its directional
counterpart, that unifies the two standard CQs.
As a result, we obtained novel and improved results for the metric subregularity constraint qualification
and we established interesting connections among the well-known CQs and the new ones.

\noindent
In the following diagram, we summarize the relations between the various constraint qualifications weaker than
GMFCQ that imply MSCQ.
The point-based conditions are naturally of primary interest
and are hence emphasized in double-framed boxes.
Note that pseudo- and quasi-normality are included as special cases
of PQ-normality for $\mli^P$ and $\mli^Q$.

\unitlength1mm
\begin{figure}[!h]
\begin{picture}(145,75)
\put(35,64){\framebox(40,11){}}
\put(35.5,64.5){\framebox(39,10){\parbox{36mm}{GMFCQ/Mord. crit.:\\
$\Lambda^0(\xb) = \{0\}$}}}
\put(45,62){\vector(-3,-1){25}}
\put(45,62){\vector(0,-1){10}}
\put(45,62){\vector(3,-1){25}}
\put(0,41){\framebox(30,11){}}
\put(0.5,41.5){\framebox(29,10){\parbox{27mm}{FOSCMS: $u \neq 0$:\\$\Lambda^0(\xb;u) = \{0\}$}}}
\put(35,44){\framebox(20,6){PQ-N($\mli$)}}
\put(70,41){\framebox(75,11){}}
\put(70.5,41.5){\framebox(74,10){\parbox{73mm}{SOSCPQN($\mli$): $0\neq\lambda \in \Lambda^0(\xb), u \neq 0, w$:\\
$\min_{\nu} \big( \ip{\lambda_{\nu}}{\nabla F_{\nu}(\xb) w} + u^T \nabla^2 \ip{\lambda_{\nu}}{F_{\nu}}(\xb) u \big) < 0$}}}
\put(69,46){\vector(-1,0){13}}
\put(57,47){\scriptsize Ass.\ref{Ass}($\mli$)}
\put(10,39){\vector(0,-1){10}}
\put(10,39){\vector(3,-1){25}}
\put(45,41){\vector(-3,-1){33}}
\put(80,39){\vector(-3,-1){25}}
\put(0,21){\framebox(20,6){Dir.PQ-N($\mli$)}}
\put(35,18){\framebox(75,11){}}
\put(35.5,18.5){\framebox(74,10){\parbox{73mm}{SOSCdirPQN($\mli$): $\norm{u}=1,0\neq\lambda \in \Lambda^0(\xb;u),w$:\\
$\min_{\nu} \big( \ip{\lambda_{\nu}}{\nabla F_{\nu}(\xb) w} + u^T \nabla^2 \ip{\lambda_{\nu}}{F_{\nu}}(\xb) u \big) < 0$}}}
\put(34,23){\vector(-1,0){13}}
\put(22,24){\scriptsize Ass.\ref{Ass}($\mli$)}
\put(10,19){\vector(0,-1){10}}
\put(0,1){\framebox(20,6){MSCQ}}
\end{picture}
\caption{Constraint qualifications for GMP \eqref{eq:GMP}.}
\end{figure}
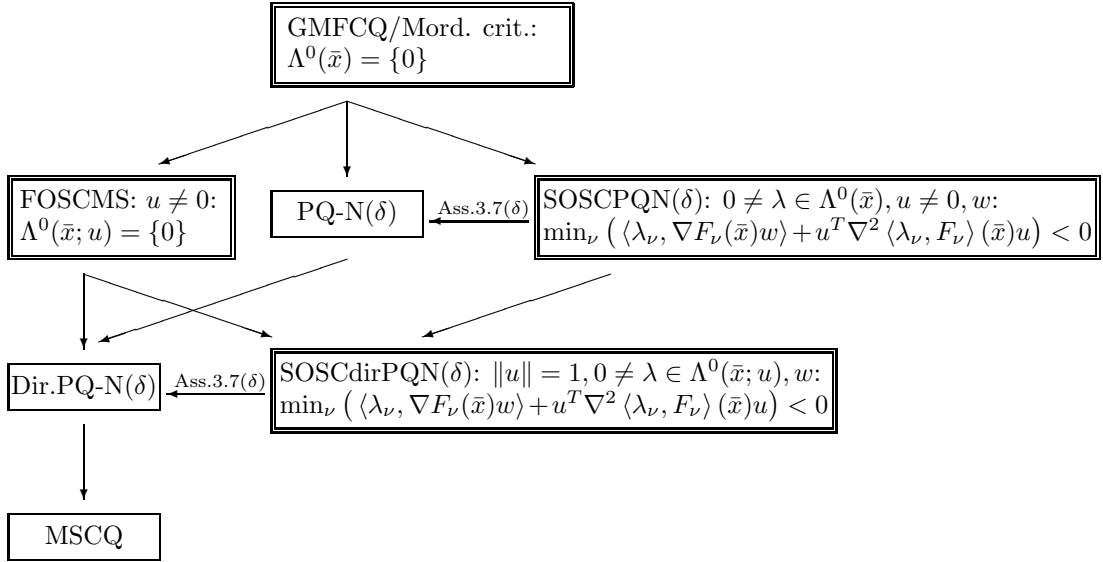 

%% file: S4-DS_v2.tex
\section{Programs with disjunctive constraints}\label{sec:Disjunctive Sets}

\noindent
In this section we study a special case of problem \eqref{eq:GMP}
in which the set $\Gamma$ is {\em disjunctive}, that is it can be written as a union of finitely many convex polyhedra, i.e.,
\begin{equation}\label{eq:Disjunctive Sets}
\Gamma=\bigcup_{\ell=1}^N \Gamma^{\ell} \quad \text{with}\quad \Gamma^{\ell} \subset \R^d\;\text{convex polyhedral},
\end{equation}
where we refer the reader to Section \ref{sec:Poly} for a definition of convex polyhedral sets.
Subsequently, we call problem \eqref{eq:GMP} with $\Gamma$ disjunctive (in the sense of \eqref{eq:Disjunctive Sets}) as a
{\em (mathematical) program with disjunctive constraints}
or simply a {\em disjunctive program}.

Disjunctive programs have been systematically studied for decades, see, e.g., \cite{St03} and the references therein.
For more recent works on disjunctive programs, which are also more related to our approach,
we refer to the papers \cite{BeGfr16d,FleKanOut07,Gfr14a,Me19} and the thesis \cite{Be17}.

The most prominent examples of disjunctive programs are  the aforementioned classes of MPCCs, MPVCs,
as well as {\em mathematical programs with relaxed cardinality constraints} (MPrCCs), {\em mathematical programs with relaxed probabilistic constraints} (MPrPCs),
and the recently introduced
{\em mathematical programs with switching constraints} (MPSCs).

For the mathematical background and several applications we refer the reader to the textbooks
\cite{LuPaRa96,OutKoZo98} for MPCCs as well as to the book \cite{De02} on the closely related class of bilevel programs.
As for MPVCs we refer to the paper
\cite{AchKa08} and the thesis \cite{Ho09} and the references therein.
For relaxed cardinality constrained problems we point to the papers \cite{BKS16,CerKanSchw16}. 
For MPrPCs see \cite{AB16}, and for MPSCs see \cite{Me18}. 

Dropping standard constraints for brevity, all of these programs exhibit the general form
\begin{equation} \label{eq : MPsDCs}
 \min_{x \in \R^n} f(x) \quad  \st \quad (G_i(x),H_i(x)) \in \widetilde\Gamma \ (i \in V),
\end{equation}
where $f,G_i,H_i: \R \to \R$ are continuously differentiable,
$V$ is a finite index set and $\widetilde\Gamma$
is given by
\begin{itemize}
\item[(a)] (complementarity constraints)
\[
\widetilde\Gamma:=\Gamma_{\text{CC}}:=\set{(a,b)}{ab=0, \; a,b\geq 0}=(\R_+\times \{0\}) \cup (\{0\}\times \R_+);
\]

\item[(b)] (vanishing constraints)
\[
\widetilde\Gamma:=\Gamma_{\text{VC}}:=\set{(a,b)}{ab\leq 0, \; b\geq 0}=(\R_-\times \R_+) \cup (\R_+\times \{0\});
\]

\item[(c)] (relaxed cardinality constraints)
\[
\widetilde\Gamma:=\Gamma_{\text{rCC}}:= \set{(a,b)}{ab=0, \;b\in [0,1]}= (\R\times \{0\}) \cup (\{0\}\times [0,1]);
\]

\item[(d)] (relaxed probabilistic constraints)
\[
\widetilde\Gamma:=\Gamma_{\text{rPC}}:= \set{(a,b)}{ab\leq 0, \;b\in [0,1]}= (\R_-\times [0,1]) \cup (\R_+\times \{0\});
\]

\item[(e)] (switching constraints)
\[
\widetilde\Gamma:=\Gamma_{\text{SC}}:= \set{(a,b)}{ab=0}= (\R\times \{0\}) \cup (\{0\}\times \R).
\]
\end{itemize}
Clearly, $\Gamma_{\text{CC}}$, $\Gamma_{\text{VC}}$, $\Gamma_{\text{rCC}}$, $\Gamma_{\text{rPC}}$ and
$\Gamma_{\text{SC}}$ are  disjunctive, rendering the resulting optimization problem a disjunctive program. We point out that there is  generally not  a unique way to write the  disjunctive sets in (a)-(e)  as a union of convex polyhedral sets. For instance, $\Gamma_{\text{VC}}$ can be alternatively written as
$\Gamma_{\text{VC}} = (\R_-\times \R_+) \cup (\R \times \{0\})$.

The main finding of this section is to show that the crucial Assumption \ref{Ass} is automatically fulfilled
for disjunctive programs. In addition, we also prove that directional pseudo-normality
does not only imply, but is, in fact, equivalent to its simplified form
from Theorem \ref{The:DirModCQ},
which suggests that our sufficient conditions are not too restrictive.
Recall that Example \ref{Ex:DirVerDiff} shows that, in general, the simplified form is strictly stronger.
For these purposes, we commence our study with a preliminary section on the variational geometry of convex polyhedral sets and how these extend to a more general setting.

\subsection{Key properties of convex polyhedral sets}\label{sec:Poly}

Recall that a set is said to be {\em convex polyhedral} (or a {\em convex polyhedron}) if it is the intersection of finitely many closed half-spaces.  In particular, for a convex polyhedron $P\subset \R^s$ there exist $p\in \N$ and $ a_j\in \R^s,\;\beta_j\in \R\;(j=1,\dots,p)$ such that
\[
P=\set{y}{\ip{a_j}{y}\leq \beta_j\;(j=1,\dots,p)}.
\]
Clearly, every convex polyhedron is closed.
Due to convexity of $P$, the  regular and limiting normal
cone to $P$ coincide with the classical normal cone of convex analysis, see \eqref{eq : convexNC}.
Given $y\in P$, we have
\[
N_P(y)=\set{\sum_{j\in J(y)}\lambda_j a_j}{\lambda_j\geq 0},
\]
where $J(y):=\set{j\in\{1,\dots,p\}}{\ip{a_j}{y}=\beta_j}$, i.e., the normal cone of $P$ at $y$ is the convex cone generated by $\set{a_j}{j\in J(y)}$, see e.g. \cite[p.~67]{HUL 01}. Therefore, there is only a finite number of different normal cones induced by a convex polyhedral set, in fact, this number is bounded by $2^p$ (as there can be at most $2^p$ active sets in $\{1,\dots,p\}$).

We will make use of two essential properties of convex polyhedra.
The first one is the well-known {\em exactness of tangent approximation}, see \cite[Exercise 6.47]{RoW 98}:
Given a convex polyhedron $P$, for any $\bar y \in P$ there exists a neighborhood $U$ of $\bar y$ such that
\begin{equation}\label{eq:PolyPropConv}
P \cap U = \big( \bar{y} + T_P(\bar{y})\big) \cap U.
\end{equation}
In particular, taking into account \cite[Exercise 6.44]{RoW 98}, one has
\[
N_{P}(\bar y)= N_{T_{P}(\bar y)}(0).
\]
The second property is closely related to Assumption \ref{Ass}
as stated in the following lemma.

\begin{lemma}\label{lem:PolyProp}
 Let $P\subset \R^s$ be closed and convex,  let
 $\{y^k\in P\}\to \bar y$ and $\{\lambda^k\in N_{P}(y^k)\}\to \bar \lambda$.
 Then there exists a subsequence $K\subset \N$ such that the following hold:
\begin{itemize}
\item[(i)] We have $\ip{\bar \lambda}{y^k - \bar y}\leq 0$ for all $k\in K$;
\item[(ii)] Moreover, if $P$ is polyhedral then $\ip{\bar \lambda}{y^k - \bar y} = 0$ for all $k\in K$.
\end{itemize}
\end{lemma}
\begin{proof}
(i) Taking the limit in $\lambda^k \in N_{P}(y^k)$ yields
$\bar \lambda \in N_{P}(\bar y)$. In particular, as $y^k\in P$ we get $\ip{\bar \lambda}{y^k-\bar y}\leq 0 \;(k\in \N)$.
\smallskip

\noindent
 (ii) Recall from the discussion above, that for a convex polyhedral set there are only finitely many different normal cones. Hence,
 there exists a subsequence $K\subset \N$ such that
 $N_{P}(y^k) \equiv \mathcal{N}$ for all $k\in K$ and some closed convex cone $\mathcal{N}$.
  Consequently, from $\lambda^k \in N_{P}(y^k)$ we obtain $\bar \lambda \in \mathcal{N}=N_{P}(y^k)$ and hence
  $\ip{\bar \lambda}{y^k - \bar y}\geq 0$
  due to convexity of $P$ and $\bar y \in P$.
\end{proof}
The above lemma immediately yields that Assumption \ref{Ass} for the multi-index $\mli^P := d$ is fulfilled at every feasible point for program \eqref{eq:GMP} with convex polyhedral $\Gamma$, regardless of the constraint mapping $F$.
However, since we are not  primarily interested in this convex polyhedral setting, we now state the desirable properties from \eqref{eq:PolyPropConv} and Lemma \ref{lem:PolyProp} (ii) in a general form.
To this end, given an arbitrary closed set $C \subset \R^d$ and $\bar y \in C$, consider the following condition:
\begin{equation}\tag{P1}\label{eq : PolyhMain}
\exists \ U(\bar y): \quad C \cap U(\bar y) = \big( \bar{y} + T_C(\bar{y})\big) \cap U(\bar y),
\end{equation}
where $U(\bar y)$ denotes a neighborhood of $\bar y$.
Moreover, given also a multi-index $\mli \in \N^l$ with $\vert \mli \vert=d$
and $\bar \lambda \in \R^d$, consider the condition:
\begin{equation} \tag{P2}\label{eq : MainSetPropPQ}
\forall \{y^k\in C\}\to \bar y,\; \{\lambda^k\in \widehat N_{C}(y^k)\}\to \bar \lambda, \;\exists \; K\subset \N:\;
  \ip{\bar \lambda_{\nu}}{y_{\nu}^k - \bar y_{\nu}} = 0 \; (\nu \in I_{\mli},\;k\in K),
\end{equation}
where $K$ is a subsequence of $\N$.
Note that \eqref{eq : MainSetPropPQ} is automatically fulfilled
if $\bar \lambda \notin N_C(\bar y)$.
We will repeatedly refer to these conditions in the subsequent study and hence we formulated it for an arbitrary multi-index $\mli$.
Clearly, if $\xb$ is feasible for \eqref{eq:GMP} and $\Gamma$ satisfies \eqref{eq : MainSetPropPQ} for $\mli$, $\bar y = F(\xb)$ and every multiplier $\bar \lambda \in N_\Gamma(F(\xb))$,
then Assumption \ref{Ass} for $\mli$ is fulfilled at $\xb$.

Motivated by the disjunctive setting in \eqref{eq:Disjunctive Sets},
for the remainder of our study  we deal with sets
generated by  unions  and, in addition, Cartesian products of convex polyhedra (see the product setting in Section 5).
Hence, we now examine properties \eqref{eq : PolyhMain} and \eqref{eq : MainSetPropPQ} under these set operations on convex polyhedra.

Consider first a collection of closed sets $C^i \subset \R^d$ for
$i = 1,\ldots,q$ and set $C:= \bigcup_{i = 1}^q C^i$.
We start with some elementary observations about tangent and normal cones.
To this end, for $y \in C$, let us denote $\mathcal{I}(y):=\set{i\in\{1,\ldots,q\}}{y\in C^i}$ and observe that, by the definition of the tangent cone, we have
\begin{equation}\label{eq:Tangent Disjunctive}
T_{C}(y)=\bigcup_{i \in \mathcal{I}(y)} T_{C^i}(y),
\end{equation}
hence, by polarization
\begin{equation}\label{eq:Frechet Normal Disjunctive}
\widehat N_{C}(y)=\bigcap_{i \in \mathcal{I}(y)} \widehat N_{C^i}(y).
\end{equation}

\noindent
This yields the following elementary estimate
\begin{equation}\label{lem:Disj Normal}
  N_{C}(y) \subset \bigcup_{i \in \mathcal{I}(y)} N_{C^i}(y),
\end{equation}
which can be derived, e.g., from the more general result \cite[Proposition 3.1]{BeGfrOut18}.

On the other hand, consider now $C=\prod_{i = 1}^{r} C_i$,
where $C_i \subset \R^{d_i}$ is closed for $i=1, \ldots, r$
and let $y=(y_1, \ldots, y_r) \in C$.
By \cite[Proposition 6.41]{RoW 98}, we have
\begin{equation} \label{NCtoProducts}
 \widehat N_{C}(y) = \prod_{i = 1}^r \widehat N_{C_i}(y_i)\AND
 N_{C}(y) = \prod_{i = 1}^r N_{C_i}(y_i).
\end{equation}
Note that for the tangent cones, \cite[Proposition 6.41]{RoW 98} in general yields
only the inclusion $T_{C}(y) \subset \prod_{i = 1}^r T_{C_i}(y_i)$.
It can be easily seen, however, that
\begin{equation} \label{TCtoProducts}
 T_{C}(y) = \prod_{i = 1}^r T_{C_i}(y_i)
\end{equation}
holds, provided $C_i$ satisfies \eqref{eq : PolyhMain} at $\bar y_i$
for all $i=1,\ldots,r$.
Indeed, for $v=(v_i) \in \prod_{i = 1}^r T_{C_i}(y_i)$ we readily
obtain from \eqref{eq : PolyhMain} for every $i=1,\ldots,r$
the existence of $\alpha_i > 0$ such that $y_i + \alpha v_i \in C_i$
holds for all $\alpha \leq \alpha_i$.
Taking $\bar \alpha := \min \alpha_i$ yields $y + \alpha v \in C$
for all $\alpha \leq \bar \alpha$ and $v \in T_{C}(y)$ follows.

Next we show that conditions \eqref{eq : PolyhMain} and \eqref{eq : MainSetPropPQ} are preserved under unions and products,
provided the obvious adjustments of multi-index, point and multiplier
are made if needed.

\begin{proposition} \label{Lem : UnionToAss}
Let $C= \bigcup_{i = 1}^q C^i$ with $C^i \subset \R^d\;(i=1,\dots,q)$ closed and let $\bar y\in C$.
\begin{itemize}
\item[(i)] If $C^i$ satisfies \eqref{eq : PolyhMain} at $\bar y$
for all $i \in \mathcal{I}(\bar y)$, then $C$ also satisfies \eqref{eq : PolyhMain} at $\bar y$.
\item[(ii)] If $C^i$ satisfies \eqref{eq : MainSetPropPQ} for some multi-index $\mli$, the point $\bar y$ and some $\bar \lambda$ for all $i \in \mathcal{I}(\bar y)$, then $C$
also satisfies \eqref{eq : MainSetPropPQ} for $\mli$, $\bar y$ and $\bar \lambda$.
\end{itemize}
\end{proposition}
\begin{proof}
Denoting $U^i(\bar y)$ for $i \in \mathcal{I}(\bar y)$
the neighborhoods given by the assumption (i)
and taking into account \eqref{eq:Tangent Disjunctive},
the first statement follows easily by setting
$U(\bar y) := \bigcap_{i \in \mathcal{I}(\bar y)} U^i(\bar y) \cap \widetilde U(\bar y)$, where $\widetilde U(\bar y)$ is a neighborhood of $\bar y$ such that
$C \cap \widetilde U(\bar y) = \bigcup_{i \in \mathcal{I}(\bar y)} C^i \cap \widetilde U(\bar y)$.
Clearly, the existence of $\widetilde U(\bar y)$ is guaranteed
by the closedness of $C^i$ $(i \notin \mathcal{I}(\bar y))$.

In order to prove (ii), consider sequences $\{y^k\in C\}\to \bar y$ and $\{\lambda^k\in \widehat N_{C}(y^k)\}\to \bar \lambda$.
From \eqref{eq:Frechet Normal Disjunctive}, closedness of $C^i$
and finiteness of $\mathcal{I}(\bar y)$ one easily obtains
the existence of $j \in \mathcal{I}(\bar y)$
and a subsequence $\tilde K \subset \N$ such that
\begin{equation*}\label{eq:Lambda_k}
\lambda^k \in \widehat N_{C^j}(y^k) \quad (k\in \tilde K).
\end{equation*}
The assumption now yields the existence of a subsequence $K \subset \tilde K$ such that
$\ip{\bar \lambda_{\nu}}{y_{\nu}^k - \bar y_{\nu}} = 0$ for $\nu \in I_{\mli}$ and $k\in K$.
\end{proof}
Recall that if $\bar \lambda \notin N_{C^i}(\bar y)$ for some $i \in \mathcal{I}(\bar y)$,
then $C^i$ automatically satisfies \eqref{eq : MainSetPropPQ}.

\begin{proposition} \label{Lem: ProdToAss}
Let $C=\prod_{i = 1}^{r} C_i$ with $C_i \subset \R^{d_i}$
$(i=1,\ldots,r)$ closed and $\bar y = (\bar y_1, \ldots, \bar y_r) \in C$.
\begin{itemize}
\item[(i)] If $C_i$ satisfies \eqref{eq : PolyhMain} at $\bar y_i$
for all $i=1,\ldots,r$, then $C$ satisfies \eqref{eq : PolyhMain} at $\bar y$.
\item[(ii)] If $C_i$ satisfies \eqref{eq : MainSetPropPQ} for multi-index $\mli_i$ with $\vert \mli_i \vert = d_i$, the point $\bar y_i$ and $\bar \lambda_i$ for all $i=1,\ldots,r$, then $C$
satisfies \eqref{eq : MainSetPropPQ} for $\mli=(\mli_1, \ldots, \mli_r)$, $\bar y$ and $\bar \lambda=(\bar \lambda_1, \ldots, \bar \lambda_r)$.
\end{itemize}
\end{proposition}
\begin{proof}
Denoting by $U_i(\bar y_i)\; (i=1,\ldots,r)$
the neighborhoods given by the assumption in (i),
the first statement follows by simply setting
$U(\bar y) := \prod_{i = 1}^r U_i(\bar y_i)$
and applying \eqref{TCtoProducts}.

In order to prove (ii), consider sequences $\{y^k\in C\}\to \bar y$ and $\{\lambda^k\in \widehat N_{C}(y^k)\}\to \bar \lambda$.
By \eqref{NCtoProducts}, we  have
$\lambda_{i}^k\in \widehat N_{C_{i}}(y_{i}^k)$
for every $i = 1, \ldots, r$ and $k\in \N$.
By assumption, there exists a subsequence
$K_1 \subset \N$ with $\ip{\bar \lambda_{1,\nu_1}}{y_{1,\nu_1}^k - \bar y_{1,\nu_1}} = 0$ $(\nu_1 \in I_{\mli_1}, k\in K_1)$.
Consequently, by assumption, there exists a subsequence $K_2\subset K_1$ such that
$\ip{\bar \lambda_{2,\nu_2}}{y_{2,\nu_2}^k - \bar y_{2,\nu_2}} = 0$ $(\nu_2 \in I_{\mli_2}, k\in K_2)$. Repeating this argument another  $r-2$ times, we find that there exists a subsequence $K(=K_r)$ such that $\ip{\bar \lambda_{i,\nu_i}}{y_{i,\nu_i}^k - \bar y_{i,\nu_i}} = 0$
$(\nu_i \in I_{\mli_i}, k\in K)$ for all $i=1,\ldots,r$. This proves the statement.
\end{proof}

We conclude this subsection by showing that  the program \eqref{eq:GMP}, with $\Gamma$ satisfying properties
\eqref{eq : PolyhMain} and \eqref{eq : MainSetPropPQ}, automatically satisfies the crucial Assumption \ref{Ass}, and  in addition,  that directional PQ-normality is equivalent to its simplified counterpart in this case. We point out that this result is the very foundation for all remaining results of the paper.

\begin{proposition}\label{Pro:EquivDir}
Let $\bar x$ be feasible for \eqref{eq:GMP} with $\Gamma$ closed and
satisfying \eqref{eq : PolyhMain} at $\bar y=F(\xb)$
as well as \eqref{eq : MainSetPropPQ} for some multi-index $\mli$,
the point $\bar y=F(\xb)$ and every multiplier
$\bar \lambda \in N_\Gamma(F(\xb))$.
Then Assumption \ref{Ass} for $\mli$ is fulfilled at $\xb$ and, moreover,
(directional) PQ-normality w.r.t. $\mli$ at $\xb$ is equivalent to
its simplified form \eqref{eq : MPQNdef} from Theorem \ref{The:ModCQ} (Theorem \ref{The:DirModCQ}).
\end{proposition}
\begin{proof}
Assumption \ref{Ass} for $\mli$ at $\bar x\in \cX $ follows from \eqref{eq : MainSetPropPQ} for  $\Gamma$  with $\mli$ at $\bar y=F(\bar x)\in \Gamma$.
Hence, the statement for the nondirectional version follows from Theorem \ref{The:ModCQ}.
Similarly, the implication from the directional simplified form to
directional PQ-normality follows from Theorem \ref{The:DirModCQ}.

It remains to show that PQ-normality w.r.t. $\mli$ in direction $u$
implies its simplified form. We do this by contraposition, so let us assume that there exists $\bar \lambda \in \Lambda^0(\xb;u)\setminus \{0\}$
and $\{x^k\} \to \bar x$ such that $(x^k - \xb)/\norm{x^k - \xb} \to u$ and
\[
\ip{\bar \lambda_{\nu}}{F_{\nu}(x^k)-F_{\nu}(\xb)} >0 \ \text{ for } \  \nu \in I_{\mli}(\bar\lambda), \ (k\in\N).
\]
By the definition of the directional normal cone, there exists $\{t_k\} \downarrow 0$ and $\{w^k\} \to \nabla F(\xb)u$ as well as
$\{\lambda^k \in \widehat N_{\Gamma}(F(\xb) + t_k w^k)\} \to \bar \lambda$.
Taking into account \eqref{eq : PolyhMain} together with
\cite[Exercise 6.44]{RoW 98} we obtain
\begin{eqnarray*}
\lambda^k & \in & \widehat N_{\Gamma}(F(\xb) + t_k w^k) = \widehat N_{F(\xb) + T_\Gamma(F(\xb))}(F(\xb) + t_k w^k) \subset \widehat N_{T_\Gamma(F(\xb))}(t_k w^k) \\
&= & \widehat N_{T_\Gamma(F(\xb))}(\alpha w^k) = \widehat N_{\Gamma -  F(\xb)}(F(\xb) + \alpha w^k -F(\xb))
\subset \widehat N_{\Gamma}(F(\xb) + \alpha w^k)
\end{eqnarray*}
for any $\alpha > 0$ sufficiently small.
Hence by setting $y^k := F(\xb) + \norm{x^k - \xb} w^k$ we conclude
$\lambda^k \in \widehat N_{\Gamma}(y^k)$.
Moreover, \eqref{eq : MainSetPropPQ} for $\mli$ yields that,
by passing to a subsequence if necessary, we may take $y^k$ such that
$\ip{\bar \lambda_{\nu}}{y_{\nu}^k - F_{\nu}(\bar x)} = 0$, for all $\nu \in I_{\mli}$ and $k \in \N$.
Consequently, we obtain
\[\ip{\bar \lambda_{\nu}}{F_{\nu}(x^k)-y_{\nu}^k} =
\ip{\bar \lambda_{\nu}}{F_{\nu}(x^k)-F_{\nu}(\xb)} >0.\]
Finally, $(y^k - F(\xb))/\norm{x^k - \xb} = w^k \to \nabla F(\xb) u$,
showing the violation of PQ-normality w.r.t. $\mli$ in direction $u$
and the proof is complete.
\end{proof}

\subsection{Pseudo-normality for disjunctive programs}

The desired results for the disjunctive setting
 \eqref{eq:Disjunctive Sets} can be viewed as a  corollary of our analysis in Section \ref{sec:Poly}.
Indeed, Lemma \ref{lem:PolyProp} and Proposition \ref{Lem : UnionToAss} yield that a disjunctive set $\Gamma$ satisfies properties \eqref{eq : PolyhMain}
and \eqref{eq : MainSetPropPQ} for the multi-index $\mli^P := d$.
In particular, due to \eqref{eq : PolyhMain},
the endeavor of computing the normal cone to disjunctive $\Gamma$ at some point can be reduced to computing the
normal cone to a union of finitely many polyhedral cones at zero, i.e.,
\[
N_{\Gamma}(\bar y)=N_{\bigcup_{\ell=1}^N T_{\Gamma^{\ell}}(\bar y)}(0) =
N_{T_{\Gamma}(\bar y)}(0),
\]
see \cite[p.~59]{HeO 08}.
More importantly, the following corollary is a consequence of Proposition \ref{Pro:EquivDir}.

\begin{corollary}\label{cor:Disj Sets}
Let $\Gamma$ be disjunctive in the sense of \eqref{eq:Disjunctive Sets}.
Then $\Gamma$ satisfies \eqref{eq : PolyhMain} at every point $\bar y \in \Gamma$
as well as \eqref{eq : MainSetPropPQ} for the multi-index $\mli^P := d$ at every point $\bar y$ and every $\bar \lambda$. In particular,
Assumption \ref{Ass} for $\mli^P$ is fulfilled at every feasible point $\xb$
for disjunctive programs.
Moreover, (directional) pseudo-normality at $\xb$ is equivalent to
its simplified form:
(for any $u \in \R^n$ with $\norm{u}=1$) there is no nonzero $\bar \lambda \in \Lambda^0(\xb)$ ($\bar \lambda \in \Lambda^0(\xb;u)$)
such that there exists a sequence $\{x^k\} \to \bar x$
(with $(x^k - \xb)/\norm{x^k - \xb} \to u$) fulfilling
\begin{equation} \label{eq : PNdef}
\ip{\bar \lambda}{F(x^k) - F(\bar x)} > 0\ 
\ (k\in\N).
\end{equation}
\end{corollary}

We emphasize that Corollary \ref{cor:Disj Sets} clarifies that the various
definitions of pseudo-normality used in the literature stem from the same concept.
In the general setting \eqref{eq:GMP}, pseudo-normality
contains the additional sequence $\{y^k\}$,
but in the special cases of disjunctive programs it reduces to the simplified version without $\{y^k\}$.

Corollary \ref{cor:Disj Sets} also allows us to use all the sufficient
conditions for pseudo-normality, hence also for MSCQ, studied in Section 3.
These conditions now take on simpler forms
since the vector optimization techniques reduce to
standard optimization in the disjunctive setting. This can be seen from
\eqref{eq : PNdef}, which yields that
pseudo-normality of $\xb$ is equivalent
to $\xb$ being a local maximizer of $\ip{\bar \lambda}{F(x)}$ for
all $\bar \lambda \in \Lambda^0(\xb)$, cf. Theorem \ref{Pro : PQNormVSMax}.
In particular, the second-order sufficient conditions from Corollary \ref{Cor : PQNormSC}
and Proposition \ref{Cor: SOSCdirPQN} read as follows.

\begin{corollary}\label{Cor:SOSCPN}
Let $\bar x$ be feasible for \eqref{eq:GMP} with $\Gamma$ disjunctive and $F$ twice differentiable at $\xb$.
Consider the following two conditions:
\begin{itemize}
 \item[(i)] second-order sufficient condition for pseudo-normality (SOSCPN):
 For every $0 \neq \bar \lambda \in \Lambda^0(\xb)$
 and every $0 \neq u \in \R^n$ one has
 \begin{equation} \label{eq:SOSCPN}
u^T \nabla^2 \ip{\bar \lambda}{F}(\xb) u < 0;
\end{equation}
 \item[(ii)] second-order sufficient condition for directional pseudo-normality (SOSCdirPN):
 For every $u \in \R^n$ with $\norm{u}=1$
 and every $0 \neq \bar \lambda \in \Lambda^0(\xb;u)$ one has \eqref{eq:SOSCPN}.
\end{itemize}
Then condition (i) (condition (ii)) implies (directional) pseudo-normality at $\xb$.
In particular, either of the two conditions implies MSCQ at $\bx$.
\end{corollary}

Clearly, an affine $F$ can never fulfill the strict inequality of SOSCPN.
The required maximality of $\xb$
expressed in \eqref{eq : PNdef} can be secured nonetheless.

\begin{corollary} \label{Cor : MPN_Flin}
 Let $\bar x $ be feasible for \eqref{eq:GMP} with $\Gamma$ disjunctive.
 If $F$ is affine then pseudo-normality, and consequently also MSCQ, holds at $\bx$.
\end{corollary}
\begin{proof} For $F$ affine we have $F(x) = F(\bar x) + \nabla F(\bar x)(x - \bar x)$ for all $x\in \R^n$.
Hence, taking into account $\bar\lambda \in \Lambda^0(\bar x)$ we find that
\[
\ip{\bar \lambda}{F(x)-F(\bar x)}=\ip{\nabla F(\bar x)^T \bar \lambda}{x-\bar x}=0,
\]
showing that $\xb$ is a local maximizer of $\ip{\bar \lambda}{F}$ and pseudo-normality thus follows.
\end{proof}

We point out  that the sufficiency of SOSCdirPN for MSCQ established in Corollary \ref{Cor:SOSCPN} corresponds
to the sufficiency of Gfrerer's SOSCMS for MSCQ (Proposition  \ref{The : CQsToMS} (iii)).
In turn, Corollary \ref{Cor : MPN_Flin} corresponds to Robinson's result (Proposition \ref{The : CQsToMS} (iv)).
Hence, by employing the notion of (directional) pseudo-normality and its sufficiency for MSCQ, we found new proofs for these interesting results.
Moreover, the notion of directional quasi-normality unifies all sufficient conditions for MSCQ from Proposition \ref{The : CQsToMS}.

Note that the analogous results were obtained also in
\cite[Theorem 4.1., Proposition 4.2.]{BYZ19}.
What was not noticed there, however,
is the underlying maximality principle \eqref{eq : PNdef},
which provides a nice understanding and makes things much simpler.
In particular, it enables us to extend the above results by means of higher-order analysis.

\subsection{Higher-order conditions}\label{sec:Higher}

In order to proceed, we rely once more on the notion of multi-indices. First, we introduce the following standard notation:
Given $\alpha=(\alpha_1,\ldots,\alpha_n) \in \N^n$ and $x=(x_1,\ldots,x_n) \in \R^n$
we set
\[\vert \alpha \vert := \alpha_1 + \ldots + \alpha_n, \quad \alpha! := \alpha_1! \ldots \alpha_n!, \quad
x^{\alpha} = x_1^{\alpha_1} \ldots x_n^{\alpha_n}.\]
Given a function $g:\R^n \to \R$, $m$-times differentiable at $\xb$,
and $\alpha \in \N^n$ with $\vert \alpha \vert \leq m$ we set
\[D^{\alpha} g (\xb)=\frac{\partial^{\vert \alpha \vert} g (\xb)}{\partial x_1^{\alpha_1} \ldots \partial x_n^{\alpha_n}}.\]

\begin{corollary} \label{Cor:MTHorder}
 Let $\bar x$ be feasible for a disjunctive program with $F$  $m$-times differentiable at $\xb$.
 Consider the following two conditions:
 \begin{itemize}
 \item[(i)] for every $0 \neq \bar \lambda \in \Lambda^0(\xb)$,
$1\leq q < m$, $w \in \R^n$ and all $0\neq u \in \R^n$ one has
\begin{equation} \label{eq : NthOrderSCPN}
 \sum_{\vert \alpha \vert = q} \frac{D^{\alpha} \ip{\bar \lambda}{F}(\xb)}{\alpha!} w^{\alpha} \leq 0
 \quad \textrm{ and } \quad
 \sum_{\vert \alpha \vert = m} \frac{D^{\alpha} \ip{\bar \lambda}{F}(\xb)}{\alpha!} u^{\alpha} < 0;
\end{equation}
 \item[(ii)] for every $u \in \R^n$ with $\norm{u}=1$, $0 \neq \bar \lambda \in \Lambda^0(\xb;u)$,
$1\leq q < m$ and all $w \in \mathcal{U}$,
where $\mathcal{U}$ denotes a neighbourhood of $u$,
one has \eqref{eq : NthOrderSCPN}.
\end{itemize}
Then condition (i) (condition (ii)) implies (directional) pseudo-normality at $\xb$.
In particular, either of the two conditions implies MSCQ at $\bx$.
\end{corollary}
\begin{proof}
 Both statements follows from the same arguments,
 namely, given $0 \neq \bar \lambda \in \Lambda^0(\xb)$ and $1\leq q < m$ and
 setting $u_k:=(x^k - \xb)/\norm{x^k - \xb}$, Taylor expansion
 together with \eqref{eq : NthOrderSCPN} yield
 \begin{eqnarray*}
  \ip{\bar \lambda}{F(x^k)-F(\xb)} &=&
  \sum_{1 \leq \vert \alpha \vert \leq m} \frac{D^{\alpha} \ip{\bar \lambda}{F}(\xb)}{\alpha!} (x^k - \xb)^{\alpha} + o(\norm{x^k - \xb}^m)\\
  & \leq & \norm{x^k - \xb}^m \Big( \sum_{\vert \alpha \vert = m} \frac{D^{\alpha} \ip{\bar \lambda}{F}(\xb)}{\alpha!} u_k^{\alpha} + o(1) \Big) < 0.
 \end{eqnarray*}
\end{proof}

Similarly as in the case of affine $F$, the strict inequality of the above higher-order sufficient conditions does not have to be fulfilled, as long as $F$ has polynomial structure,
i.e., for every $i=1,\ldots,d,$ and every $x,$ we have
\begin{equation}\label{eq:Polyn}
F_i(x) = \sum_{\vert \alpha \vert \leq m} c_{i,\alpha} x^{\alpha}
\end{equation}
for some $m \in \N$, denoting the degree of $F$, and $c_{i,\alpha} \in \R$.
We point out that one actually has
$c_{i,\alpha} = D^{\alpha} F_i(0)/\alpha!$
and \eqref{eq:Polyn} can be equivalently rewritten as
\begin{equation}\label{eq:Polyn2}
F_i(x) = \sum_{\vert \alpha \vert \leq m} \frac{D^{\alpha} F_i(\xb)}{\alpha!} (x - \xb)^{\alpha}
\end{equation}
for arbitrary $\bar x \in \R^n$.
\begin{corollary} \label{Cor:polyn}
 Let $\bar x$ be feasible for a disjunctive program with $F$ being polynomial of degree $m$,
 i.e., given by \eqref{eq:Polyn}.
 Consider the following two conditions:
 \begin{itemize}
 \item[(i)] for every $0 \neq \bar \lambda \in \Lambda^0(\xb)$,
 $1 \leq q \leq m,$ and for all $w \in \R^n$ one has
 \begin{equation} \label{eq : PolynSCPN}
  \sum_{\vert \alpha \vert = q} \frac{D^{\alpha} \ip{\bar \lambda}{F}(\xb)}{\alpha!} w^{\alpha} \leq 0;
 \end{equation}
 \item[(ii)] for every $u \in \R^n$ with $\norm{u}=1$, $0 \neq \bar \lambda \in \Lambda^0(\xb;u)$,
$1\leq q < m$ and all $w \in \mathcal{U}$,
where $\mathcal{U}$ denotes a neighbourhood of $u$,
one has \eqref{eq : PolynSCPN}.
\end{itemize}
Then condition (i) (condition (ii)) implies (directional) pseudo-normality at $\xb$.
In particular, either of the two conditions implies MSCQ at $\bx$.
\end{corollary}
\begin{proof}
 Denoting $c_{\alpha}:=(c_{1,\alpha}, \ldots, c_{d,\alpha})$ and taking into account \eqref{eq:Polyn2},
 for any $\bar \lambda \neq 0$, one has
\[\ip{\bar \lambda}{F(x)} = \sum_{\vert \alpha \vert \leq m} \ip{\bar \lambda}{c_{\alpha}} x^{\alpha} = \sum_{\vert \alpha \vert \leq m} \frac{D^{\alpha} \ip{\bar \lambda}{F}(\xb)}{\alpha!} (x - \xb)^{\alpha}\]
 for every $x$. Hence, given $0 \neq \bar \lambda \in \Lambda^0(\xb)$
 and $1\leq q \leq m$, both statements follows from \eqref{eq : PolynSCPN} since
 \begin{equation*}
  \ip{\bar \lambda}{F(x)-F(\xb)} =
  \sum_{1 \leq \vert \alpha \vert \leq m} \frac{D^{\alpha} \ip{\bar \lambda}{F}(\xb)}{\alpha!} (x - \xb)^{\alpha} \leq 0.
 \end{equation*}
\end{proof}

Of course the above higher-order conditions are sufficient for pseudo-normality
and MSCQ also for general programs \eqref{eq:GMP} fulfilling Assumption \ref{Ass} for $\mli^P$.

\subsection{Summary and example for the disjunctive case}

For the sake of completeness, we summarize the sufficient conditions
for pseudo-normality and MSCQ in the disjunctive setting in the following theorem.

\begin{theorem}[Sufficient conditions for pseudo-normality and MSCQ]\label{cor:Ex Pen PN}
Consider \eqref{eq:GMP} with $\Gamma$ disjunctive in the sense of \eqref{eq:Disjunctive Sets} and a feasible point $\bar x$.
Then any of the conditions from
Corollaries \ref{cor:Disj Sets}, \ref{Cor:SOSCPN}, \ref{Cor : MPN_Flin}, \ref{Cor:MTHorder} and \ref{Cor:polyn}
implies (directional) pseudo-normality and MSCQ at $\bar x$.
\end{theorem}


\noindent
The following parametric example  demonstrates the usefulness of our conditions based on pseudo-normality.  

\begin{example}\label{ExampleUltimate}
  Let $\Gamma \subset \R^3$ be given by $\Gamma:=\R \times \{y \in \R^2 \mv y_2 \leq -\vert y_1 \vert\}$, $F: \R^2 \to \R^3$ defined by $F(x) := (x_1,x_2,ax_1^2 + bx_1^4 + cx_2^2 + dx_2^4)^T$
for some parameters $a,b,c,d \in \R$ and let $\xb := (0,0)$.
Clearly,
\[\nabla F(\xb)=\left(\begin{array}{ccc}
                              1 & 0 & 0\\
                              0 & 1 & 0
                              \end{array}\right)^T, \quad
\nabla^2 \ip{\lambda}{F}(\xb)=2\lambda_3\left(\begin{array}{cc}
                              a & 0\\
                              0 & c
                              \end{array}\right),
\]
and for any $\lambda=(\lambda_1,\lambda_2,\lambda_3) \in \Lambda^0(\xb) = \Lambda^0(\xb;(\pm1,0)^T) = \R_+ (0,0,1)^T$.
Note also that $\Lambda^0(\xb;u)=\emptyset$ for all directions
$u \neq (\pm1,0)^T$ with $\norm{u} = 1$ since $T_{\Gamma}(F(\xb)) = \Gamma$
and $\nabla F(\xb) u = (u_1,u_2,0)^T$. Moreover, observe that
$
\mathcal X=F^{-1}(\Gamma)=\set{x\in\R^2}{ax_1^2+bx_1^4+cx_2^2+dx_2^4\leq -|x_2|}.
$

The most crucial parameter is $a$. Indeed,
if $a>0$, then locally around $\bar x=0$, the set $ \mathcal X$ is the singleton $\{0\}$, and thus    it  can be seen  that sequence $\{x_k := (1/k,0)^T\}$ shows violation of MSCQ.
On the other hand, if $a<0$, MSCQ holds and can be verified
by SOSCMS.
Hence, suppose now that $a=0$.

Next, let us look into parameter $b$.
If $b>0$, $\{x_k := (1/k,0)^T\}$ again shows violation of MSCQ
regardless of other parameters.
Note that if $b \leq 0 < c$,
the sequence $\{\tilde x_k := (1/k^2,1/k^3)^T\}$
satisfies $\tilde x_k/\norm{\tilde x_k} \to (1,0)^T=:\bar u$, but
for $\bar \lambda := (0,0,1)^T \in \Lambda^0(\xb;\bar u)$ we get
\[
\ip{\bar \lambda}{F(\tilde x^k) - F(\bar x)}
= b/(k^8) + c/(k^6) + d/(k^{12}) > 0
\]
for sufficiently large $k$, showing violation
of pseudo-normality in direction $\bar u$, hence  we cannot use any of
the  stronger  conditions to verify MSCQ.
Clearly, a similar problem occurs if $b=c=0<d$. 

We conjecture  that MSCQ holds in this case,
but as the direct proof  appears fairly technical
and since for our purposes it is more interesting to see
the limitations of sufficient conditions in this case
(rather than determine if MSCQ holds),
we skip the details  for $b<0$
and only prove MSCQ in the simpler case $b=0$ below.

Let us mention, however,
that if $b < 0 \geq c$, we may use the directional
version of the fourth-order sufficient condition based on 
Corollary \ref{Cor:MTHorder} (ii) to verify MSCQ, even if $d>0$.

Next, we prove that MSCQ holds if $b=0$,
regardless of parameters $c$ and $d$.
Since the feasible set $\cX$, locally around $\bar x=(0,0)$, equals $\R \times \{0\}$,
we get $\dist_\cX(x)=\vert x_2 \vert$ for any $x \in \R^2$ close enough to $\bar x$.
On the other hand,
for $y,\tilde y \in \R^3$ with $y_2 = \tilde y_2$
and $y_3 \leq \tilde y_3$ we clearly have
$\dist_{\Gamma}(y) \leq \dist_{\Gamma}(\tilde y)$.
Given $\varepsilon \in (0,1) $, let $x$ be sufficiently
close to $0$ so that
$- \varepsilon \vert x_2 \vert \leq c x_2^2 + d x_2^4$.
One computes that
\[\dist_{\Gamma}(x_1,x_2,- \varepsilon \vert x_2 \vert)
= \frac{1 - \varepsilon}{\sqrt{2}} \vert x_2 \vert.\]
Thus, setting $\kappa := \frac{\sqrt{2}}{1 - \varepsilon}$ yields
\[\dist_\cX(x)=\vert x_2 \vert
= \kappa \dist_{\Gamma}(x_1,x_2,- \varepsilon \vert x_2 \vert)
\leq \kappa \dist_{\Gamma}(F(x))\]
for all $x\in \R^2$ close enough to $\bar x$, 
and hence MSCQ follows.

In order to better illustrate  the results of this paper,
in the following tables corresponding to $a = 0 > b$ and $a=0=b$,
respectively, we provide  sufficient conditions
ensuring MSCQ for given parameters. Recall from above that, for   $c>0$ and $b\leq 0$,
pseudo-normality-based  conditions are not applicable, and thus we restrict ourselves to the case 
$c\leq 0$.
As mentioned above, the case $a = 0 > b$ can be handled
by the directional fourth-order sufficient condition
while for the case $a = 0 = b$ we provided the direct proof.
We point out, however, that in both cases, unless $c=0<d$,
 one can use also other sufficient conditions as indicated in the table.
In particular, one can see the meaning of parameter $d$,
which does not seem to influence the validity of MSCQ, but it influences which sufficient conditions
can be invoked to verify it.
Note also that if $a < 0$, depending on other parameters,
conditions other that SOSCMS can be used as well.
Nevertheless, the only condition that can never be used
in case $a=0$ is SOSCPN, which is applicable if $a,c < 0$.
Hence, we further detail only the case $a=0$.

\begin{table}[h]
\centering
\begin{subtable}{.5\textwidth}
\centering
\begin{tabular}{|c|c|c|c|}
\hline
 & $c=0$ & $c<0$ \\
\hline
$d=0$ & Polyn. $4^{th}$-OSC & Polyn. $4^{th}$-OSC \\
\hline
$d>0$ & Dir. $4^{th}$-OSC & Pseudo-normality \\
\hline
$d<0$ & $4^{th}$-OSC & $4^{th}$-OSC \\
\hline
\end{tabular}
\caption{$a=0>b$}
\end{subtable}%
\begin{subtable}{.5\textwidth}
\centering
\begin{tabular}{|c|c|c|c|}
\hline
 & $c=0$ & $c<0$ \\
\hline
$d=0$ & Robinson SC & Polyn. $2^{nd}$-OSC \\
\hline
$d>0$ & Def. & Pseudo-normality \\
\hline
$d<0$ & Polyn. $4^{th}$-OSC & Polyn. $4^{th}$-OSC \\
\hline
\end{tabular}
\caption{$a=0=b$}
\end{subtable}
\caption{
\emph{Polyn.} $4^{th}$-\emph{OSC} and \emph{Polyn.} $2^{nd}$-\emph{OSC}
refer to the sufficient condition for polynomial $F$
of fourth- and second-order, respectively
(Corollary \ref{Cor:polyn} (i)),
(\emph{Dir.}) $4^{th}$-\emph{OSC} stands for (the directional version of) the fourth-order sufficient condition based on Corollary \ref{Cor:MTHorder},
\emph{Robinson SC} refers to Robinson's result (Proposition \ref{The : CQsToMS} (iv)),
and \emph{Def.} means the direct proof from definition.\\
We excluded the case $c > 0$ since validity of MSCQ either remains undetermined
($b < 0$) or was proven directly ($b=0$).
}
\end{table}

To illustrate the difference between directional
and non-directional approach, observe that
the mildest non-directional sufficient condition for MSCQ,
pseudo-normality, characterized by the maximality condition
\eqref{eq : PNdef}, is satisfied if and only if
$a,c \leq 0$ and $a < 0$ provided $b > 0$
and $c < 0$ if $d > 0$.
On the other hand, on top of the above situations,
directional pseudo-normality can be applied whenever $a < 0$
or also in case $a = 0 > b$ and $c \leq 0$.

The power of our new sufficient conditions is nicely
demonstrated for $a=0$, when Gfrerer's SOSCMS can never be used.
Similarly, Robinson's result can not be applied
unless all the parameters are zero.
\end{example}

%% file: S5-OrthoDisjSets_v2.tex
\section{Disjunctive programs with product structures}\label{sec:DisjProd}

The simplified form of quasi-normality is not sufficient for metric subregularity even in case the set $\Gamma$
under consideration is a general convex polyhedral set, see Example \ref{ex:MSviolUnderMQN}.
On the other hand, we realize that the set $\widetilde\Gamma$ in all cases \eqref{eq : MPsDCs} (a)-(e)
is a union of {\em products} of closed intervals.
This additional product structure motivates our study of {\em ortho-disjunctive programs} in Section \ref{sec:OrthoDisj},
which enables us to recover and extend several known quasi-normality results for MPCCs and MPVCs
and obtain new corresponding results for MPSCs, MPrCCs and MPrPCs.

In order to clarify the role of product structures in a broader context, consider first an instance of GMP \eqref{eq:GMP}, where
\begin{equation}\label{eq:GamProd}
\Gamma=\prod_{\nu \in I_{\mli}} \Gamma_\nu, \quad \Gamma_\nu = \bigcup_{\ell=1}^{N_\nu} \Gamma_\nu^\ell, \quad \Gamma_\nu^\ell\; \text{convex polyhedral},
\end{equation}
for some multi-index $\mli \in \N^l$ with $l \in \{1,\ldots,d\}$
and $\vert \mli \vert = d$,
i.e., $\Gamma$ is the Cartesian product of disjunctive sets.
Note that all the prototypical disjunctive programs from \eqref{eq : MPsDCs} (a)-(e) exhibit such ``outer'' product structure.

We emphasize that $\Gamma$ given by \eqref{eq:GamProd} is still a disjunctive set in the sense of \eqref{eq:Disjunctive Sets}.
Indeed, denoting $\mathcal{J}:=\prod_{\nu \in I_{\mli}} \{1,\ldots,N_\nu\}$, for $\vec{\boldsymbol{\ell}} \in \mathcal{J}$
the set $\Gamma^{\vec{\boldsymbol{\ell}}} := \prod_{\nu \in I_{\mli}} \Gamma_\nu^{\ell_\nu}$ is convex
polyhedral and $\Gamma= \bigcup_{\vec{\boldsymbol{\ell}} \in \mathcal{J}} \Gamma^{\vec{\boldsymbol{\ell}}}.$
Regardless, it turns out to be advantageous to exploit the underlying product structure of $\Gamma$ rather than just treating $\Gamma$ as a disjunctive set.
One of the reasons is that we deal with the unions of only $N_{\nu}$ sets,
which is typically a small number ($N_\nu = 2$ for all $\nu$ in all cases \eqref{eq : MPsDCs} (a)-(e)),
instead of dealing with the union of $\vert \mathcal{J} \vert = \prod_{\nu \in I_{\mli}} N_\nu$ sets.
We point out that the newly developed concept of $\mathcal{Q}$-stationarity from \cite{BeGfr16b,BeGfr16d} takes advantage of this observation.

On the basis of Propositions \ref{Lem : UnionToAss} and \ref{Lem: ProdToAss} we readily infer
that, on top of property \eqref{eq : PolyhMain},
$\Gamma$ given by \eqref{eq:GamProd} satisfies also \eqref{eq : MainSetPropPQ} {\em for multi-index $\mli$}.
Proposition \ref{Pro:EquivDir} thus yields that in this case the (directional) PQ-normality w.r.t. $\mli$
coincides with its simplified form.
In particular, standard NLPs, where $\Gamma=\{0\}^r \times \R_{-}^{d-r}$ for some $r \leq d$,
fit into \eqref{eq:GamProd} with the multi-index $\mli^{Q}:=(1,\ldots,1) \in \N^d$
and hence we can readily handle  quasi-normality for NLPs with ease.

Utilizing the ``outer'' product structure on its own, however,
does not enable one to analyze the quasi-normality for programs from \eqref{eq : MPsDCs} (a)-(e),
where the factors $\Gamma_\nu = \widetilde\Gamma$ are two-dimensional.
To overcome this, consider the GMP \eqref{eq:GMP} with the ``inner'' product structure, i.e., where
\begin{equation} \label{eq:GamProd2}
 \Gamma = \bigcup_{\ell=1}^{N} \Gamma^{\ell}, \ \Gamma^{\ell}=\prod_{\mu \in I_{\mli}} \Gamma_{\mu}^{\ell}, \ \Gamma_{\mu}^{\ell} \; \text{convex polyhedral},
\end{equation}
for some multi-index $\mli \in \N^l$.

By the same arguments as before, $\Gamma$ again satisfies \eqref{eq : MainSetPropPQ} for $\mli$ and PQ-normality w.r.t. $\mli$ attains the simplified form.
Moreover, the choice of multi-index $\mli^{Q}:=(1,\ldots,1) \in \N^d$ now offers richer setting.

\subsection{Ortho-disjunctive constraints and quasi-normality}\label{sec:OrthoDisj}

Motivated by the above discussion, we now introduce the new subclass of disjunctive programs
containing the ``inner'' product structure with {\em one-dimensional} factors.
To this end, consider the mathematical program of the form
\begin{equation}\label{eq:Ortho Disj 1}
\min\limits_{x \in \mathbb{R}^n} \, f(x) \quad  \st \quad F(x) \in \Gamma=\bigcup_{\ell=1}^{N} \Gamma^{\ell},\
\Gamma^{\ell}=\prod_{i \in I} [a_{i}^{\ell},b_{i}^{\ell}],
\end{equation}
where $I=\{1,\ldots,d\}$, $a_{i}^{\ell},b_{i}^{\ell} \in \R$ with
$a_{i}^{\ell} \leq b_{i}^{\ell}$ and we also allow symbols $a_{i}^{\ell} = -\infty$ and $b_{i}^{\ell} = +\infty$
to include unbounded intervals. Note that we do not work with extended real numbers, i.e., given $a \in \R$, $[a,\infty]$ stands for
$\{x \in \R \mv x\geq a\}$.
This simply means that $\Gamma^{\ell}$ is a product of closed convex subsets of $\R$, i.e., closed intervals.
We refer to such sets $\Gamma$ as {\em ortho-disjunctive} and to such programs as
{\em mathematical programs with ortho-disjunctive constraints} or briefly {\em ortho-disjunctive programs}.

Naturally, one can combine the ``inner'' and ``outer'' products and consider the Cartesian product of ortho-disjunctive sets,
a setting that indeed fits the problem class   \eqref{eq : MPsDCs} best.
As before, it can be easily shown that such sets are still ortho-disjunctive.
Moreover, only the ``inner'' products are important for our remaining analysis.
Hence, we proceed without the ``outer'' product, which is also more consistent with the notion of disjunctive sets.

On the basis of Propositions \ref{Lem : UnionToAss}, \ref{Lem: ProdToAss} and \ref{Pro:EquivDir}
we obtain the following analogon of Corollary \ref{cor:Disj Sets}.
\begin{corollary} \label{Pro : QusinormEquiv}
Set $\Gamma$ given by \eqref{eq:Ortho Disj 1} satisfies \eqref{eq : PolyhMain} at every point $\bar y \in \Gamma$
as well as \eqref{eq : MainSetPropPQ} for multi-index $\mli^Q := (1,\ldots,1)\in \N^d$ at every $\bar y$ and every $\bar \lambda$.
In particular, for ortho-disjunctive program \eqref{eq:Ortho Disj 1}, Assumption \ref{Ass} for $\mli^Q$ is fulfilled at every feasible point $\bar x$ and, moreover,
the (directional) quasi-normality at $\xb$ is equivalent to its simplified form:
(for any $u \in \R^n\setminus\{0\}$) there is no nonzero $\bar \lambda \in \Lambda^0(\xb)$ ($\bar \lambda \in \Lambda^0(\xb;u)$)
such that there exists a sequence $\{x^k\} \to \bar x$
(with $(x^k - \xb)/\norm{x^k - \xb} \to u$) fulfilling
\begin{equation} \label{eq : QNsimplified}
\bar \lambda_i \left( F_i(x^k) - F_i(\bar x) \right) > 0 \quad \ 
(i \in I(\bar\lambda) := I_{\mli^Q}(\bar\lambda) = \{i \in I \mv \bar\lambda_i \neq 0\}, \, k\in\N).
\end{equation}
\end{corollary}

Just as in the case of pseudo-normality, cf. the comments after Corollary \ref{cor:Disj Sets},
we have now clarified that, in fact, there is only one concept of quasi-normality
which, in general, contains the additional sequence $\{y^k\}$,
but in special cases, such as NLPs or MPCCs, simplifies to the known versions without $\{y^k\}$.
Moreover, the above corollary provides the definition
of quasi-normality for all other ortho-disjunctive programs.

Before we state the main result of this subsection
that parallels Theorem \ref{cor:Ex Pen PN} for pseudo-normality,
we write down explicitly the conditions from
Theorem \ref{Pro : PQNormVSMax}, Proposition \ref{Cor: SOSCdirPQN}
and Corollary \ref{Cor : PQNormSC} for multi-index $\mli^Q$ corresponding to quasi-normality.

Given $\lambda = (\lambda_i)_{i \in I}$,
$\varphi^{\lambda}$ from \eqref{eq:varphi} reads as
\begin{equation} \label{eq:MultiobjProbQN}
 \varphi^{\lambda}(x) =
\left(\lambda_i F_i(x) \right)_{i \in I(\lambda)}.
\end{equation}
Moreover, assuming that $F$ is twice differentiable at $\xb$,
the second-order sufficient conditions from Corollary \ref{Cor : PQNormSC}
and Proposition \ref{Cor: SOSCdirPQN},
respectively, read as follows:
\begin{itemize}
 \item {\em Second-order sufficient condition for quasi-normality (SOSCQN):}
 For every $0 \neq \bar\lambda \in \Lambda^0(\xb)$, every $0 \neq u \in \R^n$ with $\nabla F_i(\xb) u = 0$ for all $i \in I(\bar \lambda)$ and every $w \in \R^n$ with $\ip{w}{u} = 0$ one has
\begin{equation} \label{eq : SSOSCMQN}
\min_{i \in I(\bar \lambda)} \left( \bar\lambda_i \nabla F_i(\xb) w + u^T \nabla^2 (\bar\lambda_i F_i)(\xb) u \right) < 0;
\end{equation}
\item {\em Second-order sufficient condition for directional quasi-normality (SOSCdirQN):}
 For every $u\in \R^n$ with $\|u\|=1$,
 every $\bar\lambda \in \Lambda^0(\xb;u)$ with
 $\nabla F_i(\xb) u = 0$ for all $i \in I(\bar \lambda)$
 and every $w$ with $\ip{w}{u} = 0$ one has \eqref{eq : SSOSCMQN}.
\end{itemize}
Moreover, for a closed interval $[a,b]$ and $c \in \R$ we have
\[\dist_{[a,b]}(c) = (c-a)^- + (c-b)^+,\]
where $(q)^-:= - \min\{q,0\}$ and $(q)^+:= \max\{q,0\}$
denotes the negative and the positive part of any number $q \in \R$, respectively, extended to symbols $\pm \infty$ by the natural convention $(\infty)^- = (-\infty)^+ = 0$.
Thus, depending on which norm we consider for the products, the penalty function now reads as
\begin{eqnarray}\label{eq:PenaltyFuncOrthoDisj}
P_\alpha & = & f+ \alpha \min_{\ell=1,\ldots,N}
\dist_{\Gamma^{\ell}}\circ F \\ \nonumber
& =& \begin{cases}
      f+ \alpha \min_{\ell=1,\ldots,N}
      \sum_{i \in I} \left((F_i(\cdot) - a_{i}^{\ell})^- + (F_i(\cdot) - b_{i}^{\ell})^+ \right) & (l_1\textrm{-norm}),\\
      f+ \alpha \min_{\ell=1,\ldots,N}
      \max_{i \in I} \left((F_i(\cdot) - a_{i}^{\ell})^- + (F_i(\cdot) - b_{i}^{\ell})^+ \right) & (l_{\infty}\textrm{-norm})
     \end{cases}\quad (\alpha>0).
\end{eqnarray}

\begin{theorem}[Sufficient conditions for quasi-normality and MSCQ]\label{cor:Ex Pen QN}
Consider an ortho-disjunctive program \eqref{eq:Ortho Disj 1} and a feasible point $\bar x$.
Then each of the following conditions
implies (directional) quasi-normality and MSCQ at $\bar x$:
(i) the weak efficiency of $\xb$ for $\varphi^{\lambda}$ from \eqref{eq:MultiobjProbQN}, (ii) SOSCQN from \eqref{eq : SSOSCMQN} (SOSCdirQN).
\end{theorem}

Let us briefly comment on the importance of the previous theorem
(together with Corollary \ref{Pro : QusinormEquiv}).
First, consider only the statement that
the (simplified form of) quasi-normality \eqref{eq : QNsimplified}
implies MSCQ and hence M-stationarity and exactness
of the penalty function \eqref{eq:PenaltyFuncOrthoDisj} at local minimizers.
For MPCCs, we thus recover the following results:
\cite[Theorem 3.3]{KaS 10} (quasi-normality implies M-stationarity),
\cite[Lemma 4.3 and 4.4]{KaS 10} (pseudo-normality implies MSCQ),
\cite[Theorem 4.5 and Corollary 4.6]{KaS 10} (pseudo-normality implies exactness of $l_1$ and $l_\infty$ penalty function),
as well as \cite[Theorem 3.1]{YeZ 14} (quasi-normality implies MSCQ). Similarly, for MPVCs we recover and improve
\cite[Theorem 3.1]{QHYM18} (pseudo-normality implies exactness of the penalty function)
and the fact that quasi-normality implies M-stationarity,
which is not stated in the paper,
but follows directly from \cite[Theorem 2.1 and Definition 2.3]{QHYM18}.
Moreover, to the best of our knowledge,
pseudo- and quasi-normality were not yet introduced for MPSCs, MPrCCs and MPrPCs
and all our results are hence new when applied to these problem classes.

Second, we also provide verifiable sufficient conditions
for quasi-normality, together with sufficient conditions
for pseudo-normality (higher-order conditions, polynomiality of $F$) from Section 4,
which enhances the applicability of our results.

Finally, we open a path for a refined analysis using   directional quasi-normality
as well as all the corresponding sufficient conditions (SOSCdirQN  etc.).

In order to illuminate and compare our results with the literature,
we conclude this section with application to MPCCs.
The same exercise could be executed for other classes \eqref{eq : MPsDCs} (b)-(e).
Recall that, omitting standard equality and inequality constraints, an MPCC is given as
\begin{equation*}
 \min_{x \in \R^n} f(x) \quad  \st \quad G_i(x), H_i(x) \geq 0,\,
 G_i(x) H_i(x) = 0, \ i \in V.
\end{equation*}
The constraints of MPCCs fit the general setting $F(x) \in \Gamma$ with
$F(x) := (G_i(x),H_i(x))_{i \in V}$, and
$\Gamma := \Gamma_{\text{CC}}^{\vert V \vert}$,
where $\Gamma_{\text{CC}} = (\R_+\times \{0\}) \cup (\{0\}\times \R_+)$
is clearly ortho-disjunctive.
As we mentioned, $\Gamma$ itself is also ortho-disjunctive,
but we choose to rather keep the ``outer'' product as well,
noting that the impact is only visible at the penalty function.
We point out that the standard approach to MPCCs is to consider
$\Gamma := -\Gamma_{\text{CC}}^{\vert V \vert}$ and $F(x) := (-G_i(x),-H_i(x))_{i \in V}$ in order to work with nonnegative signs of certain multipliers,
while in our case we obtain the opposite sign restrictions.

A simple computation yields that for $(G,H) \in \Gamma_{\text{CC}}$ we have
\[
N_{\Gamma_{\text{CC}}}(G,H) =
\begin{cases}
 \{0\} \times \R & \textrm{ if } G > 0 = H,\\
 \R \times \{0\} & \textrm{ if } G = 0 < H,\\
 (\R_- \times \R_-) \cup (\{0\} \times \R) \cup (\R \times \{0\})
 & \textrm{ if } G = 0 = H.
\end{cases}
\]
Hence, denoting
\begin{eqnarray*}
I^{+0}(\xb) & := & \{i \in V \mv G_i(\xb) > 0 = H_i(\xb)\},\\
I^{0+}(\xb) & := & \{i \in V \mv G_i(\xb) = 0 < H_i(\xb)\},\\
I^{00}(\xb) & := & \{i \in V \mv G_i(\xb) = 0 = H_i(\xb)\}
\end{eqnarray*}
for some feasible point $\xb$, we conclude that
$\lambda=(\lambda_i^G,\lambda_i^H)_{i \in V} \in N_{\Gamma_{\text{CC}}^{\vert V \vert}}(F(\xb))$
if and only if
\begin{equation} \label{eq: M-stat}
 \lambda_i^G = 0, \, i \in I^{+0}(\xb), \ \lambda_i^H = 0, \, i \in I^{0+}(\xb) \textrm{ and }
 \lambda_i^G,\lambda_i^H \leq 0 \textrm{ or } \lambda_i^G \lambda_i^H = 0, \, i \in I^{00}(\xb).
\end{equation}
Consequently, Corollary \ref{Pro : QusinormEquiv} yields that $\xb$ satisfies quasi-normality provided there is no nonzero
$\bar \lambda=(\bar \lambda_i^G,\bar \lambda_i^H)_{i \in V}$ fulfilling
\[0 = \sum_{i \in V}
\big( \bar \lambda_i^G \nabla G_i(\xb) + \bar \lambda_i^H \nabla H_i(\xb) \big)\]
together with \eqref{eq: M-stat} such that there exists
a sequence $\{x^k\} \to \bar x$ with
\begin{equation*}
\bar \lambda^G_i G_i(x^k) > 0 \text{ if } \ \bar\lambda_i^G \neq 0 \ \textrm{ and } \
\bar \lambda^H_i H_i(x^k) > 0 \text{ if } \ \bar\lambda_i^H \neq 0, \ (k\in\N).
\end{equation*}
On the other hand, $\xb$ is M-stationary provided there exists
$\bar \lambda=(\bar \lambda_i^G,\bar \lambda_i^H)_{i \in V}$ satisfying \eqref{eq: M-stat} and
\[0 = \nabla f(\xb) + \sum_{i \in V}
\big( \bar \lambda_i^G \nabla G_i(\xb)+ \bar \lambda_i^H \nabla H_i(\xb) \big).\]

Moreover, using first the $l_1$-norm to handle the ``outer'' product
we get
\[P_\alpha(x) = f(x)+\alpha\sum_{i \in V}
\dist_{\Gamma_{\text{CC}}}(G_i(x),H_i(x)).\]
Next, using the $l_\infty$-norm for the ``inner'' product,
for arbitrary $(G,H) \in \R^2$ we have
$\dist_{\Gamma_{\text{CC}}}(G,H) = \vert \min\{G,H\} \vert$.
Note that this agrees with the corresponding expression from \eqref{eq:PenaltyFuncOrthoDisj},
which reads as
$\min \big\{\max \{(G)^-, \vert H \vert \}, \max \{\vert G \vert, (H)^- \}\big\}$.
Consequently, we obtain
\begin{equation*}
 P_\alpha(x) = f(x)+\alpha\sum_{i \in V} \vert \min\{G_i(x),H_i(x)\} \vert.
\end{equation*}